\documentclass[11pt,a4paper]{amsart}
\usepackage[left=3.5cm,right=3.5cm,top=3.5cm,bottom=3.5cm]{geometry}
\usepackage{enumerate}
\usepackage{leftidx}
\usepackage{mathtools}
\usepackage{amsmath,amscd,amssymb,amsthm,mathrsfs,stmaryrd}
\usepackage[arrow,matrix]{xy}
\usepackage{graphicx,color,tikz}
\usepackage[colorlinks,linkcolor=blue,anchorcolor=blue,citecolor=blue,backref=page]{hyperref}

\usetikzlibrary{positioning,bending}
\usetikzlibrary{arrows.meta}

\DeclareMathOperator{\gal}{Gal}

\DeclareMathOperator{\Hom}{Hom}
 
\theoremstyle{definition}
\newtheorem{definition}{Definition}[section]

\newtheorem{remark}[definition]{Remark}
\newtheorem*{remark*}{Remark}

\theoremstyle{plain}
\newtheorem{theorem}[definition]{Theorem}
\newtheorem{corollary}[definition]{Corollary}
\newtheorem{lemma}[definition]{Lemma}
\newtheorem{proposition}[definition]{Proposition}
\newtheorem{conjecture}[definition]{Conjecture}

\newcommand{\F}{{ \mathbb F }}
\newcommand{\N}{{ \mathbb N }}
\newcommand{\Q}{{ \mathbb Q }}
\newcommand{\Z}{{ \mathbb Z }}

\def\FF{\mathcal F}

\renewcommand{\O}{{ \mathcal O }}

\def\sep{\text{sep}}

\author[A. Ferraguti]{Andrea Ferraguti}
\address{DICATAM, Università degli Studi di Brescia\\ via Branze 43\\ I-25123 Brescia, Italy}
\email{and.ferraguti@gmail.com}

\author[C. Pagano]{Carlo Pagano}
\address{Concordia University, Department of Mathematics and Statistics\\ Montreal \\ Quebec H3G 1M8\\ Canada}
\email{carlein90@gmail.com}

\title[Abelian dynamical Galois groups for unicritical polynomials]{Abelian dynamical Galois groups for unicritical polynomials}

\begin{document}

\begin{abstract}
Andrews and Petsche proposed in 2020 a conjectural characterization of all pairs $(f,\alpha)$, where $f$ is a polynomial over a number field $K$ and $\alpha\in K$, such that the dynamical Galois group of the pair $(f,\alpha)$ is abelian. In this paper we focus on the case of unicritical polynomials $f$, and more general dynamical systems attached to sequences of unicritical polynomials. 

After having reduced the conjecture to the post-critically finite case, we establish it for all polynomials with \emph{periodic} critical orbit, over any number field. We next establish the conjecture in full for all monic unicritical polynomials over any quadratic number field. Finally we show that for any given degree $d$ there exists a finite, explicit set of unicritical polynomials that depends only on $d$, such that if $f=ux^d+1$ is a unicritical polynomial over a number field $K$ that lies outside such exceptional set, then there are at most finitely many basepoints $\alpha$ such that the dynamical Galois group of $(f,\alpha)$ is abelian. 

To obtain these results, we exploit in multiple ways the group theory of the generic dynamical Galois group to force diophantine relations in dynamical quantities attached to $f$. These relations force in all cases, outside of the ones conjectured by Andrews--Petsche, a contradiction either with lower bounds on the heights in abelian extensions, in the style of Amoroso--Zannier, or with the computation of rational points on explicit curves, carried out with techniques from Balakrishnan--Tuitman and Siksek. 


\end{abstract}

\maketitle

\section{Introduction}
It was an insight of Lubin and Tate that one could explicitly construct (a complementary totally ramified piece of) the maximal abelian extension of a local field $K$, by taking what, in the parlance of arithmetic dynamics, is the arboreal field
$$K_{\infty}(f,0),
$$
where $f$ is a Lubin--Tate polynomial. Indeed they were able to equip the tree of pre-images
$$T_{\infty}(f,0)=\bigcup_{n \geq 0}f^{-n}(0),
$$
with a natural structure of (co-rank $1$) $\mathcal{O}_K$-module, preserved by the absolute Galois group $G_K$. This is the celebrated theory of Lubin--Tate formal groups and their classical application to explicit local class field theory. 

When $K$ is a number field, the only two non-trivial known instances of such abelian arboreal constructions with polynomials essentially arise from the multiplicative group: one has the pairs $(x^d,\zeta)$ and $(\pm T_d,\zeta+\zeta^{-1})$, where $\zeta$ is a root of unity and $T_d$ is the $d$-th Chebichev polynomial. If one allows more general rational maps, then more examples are available coming from elliptic curves with complex multiplication (see \cite{ferra2}).

At the same time, it is a general expectation in arithmetic dynamics that arboreal Galois groups over global fields should typically have a large and complicated structure: a precise implementation of this expectation can be found in Jones' open image conjecture in the quadratic case \cite{jones1,jones2}, a conjectural dynamical analogue of Serre's open image theorem.  As a matter of fact we shall exhibit below a precise connection between the size of the arboreal images and their abelianity. In light of this, it seems reasonable to expect that over number fields abelian arboreal representations should be a rarity, perhaps one subject to a precise classification. Such a conjectural classification was firstly proposed in Andrews--Petsche \cite{andrews}, where they conjectured that Chebichev and power polynomials should be the only sources of non-trivial abelian arboreal representations. More precisely they proposed the following. We recall that if $K$ is a field, $f\in K(x)$ is a rational function and $\alpha\in K$, we denote by $G_\infty(f,\alpha)$ the \emph{dynamical Galois group} of the pair $(f,\alpha)$, namely the Galois group of the extension $K(T_\infty(f,\alpha))/K$.
\begin{conjecture} \label{conj}
Let $K$ be a number field, $f \in K[x]$  a polynomial of degree $d \ge 2$ and $\alpha \in K$ be a non-exceptional point for $f$. Then $G_{\infty}(f,\alpha)$ is abelian
if and only if there exists a root of unity $\zeta$ such that the pair $(f, \alpha)$ is $K^{\textup{ab}}$-conjugated
to either 
$$(x^d,\zeta)
$$
or 
$$(\pm T_d(x), \zeta +\zeta^{-1}).$$
\end{conjecture}
Here a pair $(f,\alpha)$ is said to be \emph{exceptional} if the tree $T_{\infty}(f,\alpha)$ is finite. We refer to the first author's survey \cite{ferra3} for a more detailed exposition of the results obtained on this conjecture in the past years. Here we limit ourselves to the following brief overview. In the same work \cite{andrews} Andrews--Petsche managed to prove their conjecture for \emph{stable} quadratic polynomials over $\Q$, combining explicit estimates of Arakelov--Zhang intesection pairings with heights lower bounds of Amoroso--Zannier. Soon after, the present authors fully settled the case of quadratic polynomials over $\mathbb{Q}$, with a completely different method \cite{ferpag2}. Our method allowed us to bypass the assumption of stability, by exploiting a simple observation that was at the basis of a careful analysis of the lower central series of the Galois group of an arboreal field in our previous work \cite{ferpag1}: an automorphism of a binary rooted tree that swaps the two halves of the tree does not commute with all the elements that are linearly independent from it in the abelianization. This, and its suitable generalizations to more complicated trees, is, in a nutshell, what we call here the $1$-\emph{dimensionality principle}. In practical terms, it translates the vanishing of the commutator pairing into the existence of rational points on algebraic curves. 

This method allowed us also to prove that over any number field $K$ if $G_{\infty}(f,\alpha)$
is abelian, and $f$ quadratic, then the critical orbit of $f$ must be finite, that is $f$ must be PCF. This is a special case of Jones' conjecture, since one can prove that abelian closed subgroup of $\Omega_{\infty}$ have always infinite index \cite[Proposition 3.4]{ferpag2} (later, the first author, Ostafe and Zannier \cite{ferra2} established the same conclusion for more general rational functions with a completely different method). When $K=\Q$, one is quickly left only with the case of $x^2-1$, which we handled with a class field theoretic argument, combined with work of Anderson--Hamblen--Poonen--Walton \cite{anderson}. This was our alternative to the Arakelov theoretic approach used by Andrews--Petsche for $x^2-1$: as we are about to see, a completely different method shows that this case is a very special case of the case of polynomials with \emph{periodic} critical orbit.  

In this work, we make further progress on Conjecture \ref{conj}, in case $f$ is a monic \emph{unicritical polynomial}. First of all either by \cite{ferra2} or by  Theorem \ref{abelian_implies_pcf} \footnote{As it will become transparent in the discussion of Theorem \ref{main: quadratic}, the $1$-dimensionality principles is a useful tool also for PCF-polynomials to rule out abelian images, as such we need it anyways.} one is immediately reduced to the case where the unique critical orbit of $f$ is \emph{finite}, that is $f$ is PCF. 

In these cases, we are able to fully settle the case of unicritical polynomials with \emph{periodic} critical orbit, over any number field. More precisely we have the following. 
\begin{theorem} \label{main: periodic}
Let $K$ be a number field and let $a,c,\alpha$ be in $K$ with $a\ne 0$. Suppose that the orbit of $0$ under $f:=ax^d+c$ is periodic. Then $G_{\infty}(f,\alpha)$ is abelian if and only if $c=0$ and $(f,\alpha)$ is $K^{\text{ab}}$-conjugate to $(x^d,\gamma)$, where $\gamma$ is either $0$ or a root of unity. In particular, if $a=1$ and $G_\infty(f,\alpha)$ is abelian, then $c=0$ and $\alpha$ is either $0$ or a root of unity.
\end{theorem}

As promised above, we now see that $x^2-1$ can be immediately handled by Theorem \ref{main: periodic}, which provides a third proof of the already established classification over $\Q$ in the quadratic case.  

The proof of Theorem \ref{main: periodic} is inspired by a lemma in the work \cite[Lemma 1.2]{ahmad} of Ahmad--Benedetto--Cain--Carroll--Fang, that studied the Galois theory of iterates of $x^2-1$. In turn, this lemma was recently generalized in the work of the second author \cite{pag} that exploited the periodicity of the critical orbit to find exponential lower bounds on arboreal degrees in the unicritical non PCF case. Its generalization, which can be found already in \cite[Proposition 3.1]{pag}, is that in the periodic case, for each positive integer $n$, one can construct inside the arboreal field a $d^n$-th roots of $-\beta$ for each $\beta$ in the tree. Using this, one rapidly finds out that, if the critical orbit is periodic and the image is abelian, then the entire tree must consist of roots of unity, unless an immediate contradiction with Amoroso--Zannier height lower bound \cite{AZ} is met. From here it follows that the polynomial preserves the unit circle and hence must be $x^d$. Theorem \ref{main: periodic} and its proof were presented by the second author at the Arithmetic Dynamics Online Seminar in Autumn 2021; we have been been informed recently that subsequently Jones obtained a generalization to rational functions that will appear in a forthcoming preprint. 

We mention that one can provide a precise connection between the growth of arboreal images and Conjecture \ref{conj}. In fact using the present work one can easily see that abelian arboreal representations need to have topologically finitely generated image; it is not hard to show that finitely generated abelian subgroups of the group of automorphisms of the tree grow at most exponentially with the level. Hence a source of super-exponential lower bounds on the sequence of arboreal degrees would potentially be a tool to attack the remaining cases of Conjecture \ref{conj}, that is the strictly pre-periodic case. At present we do not know how to do this. 

We next establish the monic unicritical case of Andrews--Petsche in full for $\Q$ and all quadratic number fields. 
\begin{theorem} \label{main: quadratic}
Suppose that $K$ is a number field of degree at most $2$ over $\Q$. Then Conjecture \ref{conj} holds for all pairs of the form $(x^d+c,\alpha)$. More precisely, the only triples 
$$(x^d+c,\alpha,K),
$$
where $[K:\Q] \leq 2, \alpha,c \in K$ such that $G_{\infty}(x^d+c,\alpha)$ is abelian are the two infinite families
$$\bigcup_{[K:\Q] \leq 2, d \geq 2}\{(x^d,\pm 1,K),(x^d,0,K) \} $$
and
$$\bigcup_{[K:\Q] \leq 2} \{(x^2-2,\pm 1,K),(x^2-2, \pm 2,K), (x^2-2,0,K)\},
$$
along with the sporadic cases
$$\{(x^2-2,\pm \sqrt{2},\Q(\sqrt{2})\} \cup \{(x^2-2, \pm \sqrt{3},\Q(\sqrt{3}))\} \cup \{(x^2-2,(1\pm \sqrt{5})/2,\Q(\sqrt{5})\}
$$
and $\bigcup_{d \geq 2} \{(x^d,\pm \zeta_3^{\pm 1},\Q(\zeta_3))\} \cup \bigcup_{d \geq 2} \{(x^d,\pm i,\Q(i))\}$.
\end{theorem}
The proof of Theorem \ref{main: quadratic} follows immediately by combining Theorem \ref{PCF_classification}, Proposition \ref{case345}, Theorem \ref{case_6} and Theorem \ref{case_7}, along with the results of \cite{andrews} on $x^d$ and $T_d(x)$.

The key starting point is that we have here an infinite \emph{explicit} list of PCF unicritical polynomials over quadratic fields (Theorem \ref{PCF_classification}). After applying Theorem \ref{main: periodic} via Proposition \ref{case345}, one is left with the case of $x^2+i$ (Theorem \ref{case_6}) and the infinite family $x^{6k}+\zeta_3$ (Theorem \ref{case_7}). To handle these strictly preperiodic cases, we fully exploit our method to prove that abelian implies PCF: the $1$-dimensionality principle mentioned above. Another crucial ingredient, used to reduce the second family to a finite list, is Proposition \ref{necessary_cases}. These two allow us to reduce the problem into finding rational points on finitely many curves. The hardest case here is the Picard curve:
$$y^3=x^4+18x^2-27,
$$
where one looks for all the $\Q(\zeta_3,i)$-rational points. We manage to deal with this in Theorem \ref{chabauty}. We resort to the Chabauty method, importing tools from Balakrishnan--Tuitman \cite{balatu2,balatu} and Siksek \cite{siksek2,siksek}. Once this is done, we are left with the infinite family 
$$(x^{6k}+\zeta_3,\zeta),$$
where $\zeta$ is a $6$-th root of unity. Applying directly Amoroso--Zannier \cite{AZ} one can reduce to the case where $k$ is going up to $3^{14}$. Unfortunately this is practically unfeasible. To circumvent this problem, we employ their \emph{method} by picking an auxiliary prime that is $1$ modulo $3$ and between $6k$ and $12k$. With it, and with the fact that all our ramification is bounded in $6k$, their method of proof 
 gives us a lower bound on the height that grows roughly as a constant multiple of $\frac{\text{log}(k)}{k}$. Happily, in our case, this needs to be contrasted with an upper bound on the height that decays like a constant multiple of $\frac{1}{k}$. Making these constants explicit yields a contradiction unless $k \leq 36$, which is feasible and handled with one more appeal to the $1$-dimensionality principle (and Magma \cite{magma}).  

Finally, we prove that for most unicritical polynomials $f$, there are only finitely many basepoints $\alpha$ such that $G_\infty(f,\alpha)$ is abelian. This is formalized as follows. 
\begin{theorem} \label{main: finiteness}
Let $d\geq 2$ be an integer. Then there exists a finite set $U_d\subseteq \overline{\Q}$, depending only on $d$, such that if $K\subseteq \overline{\Q}$ is a number field, $u\in K\setminus U_d$ and $f=ux^d+1\in K[x]$, then there are only finitely many $\alpha\in K$ such that $G_\infty(f,\alpha)$ is abelian.
\end{theorem}
The set of polynomials coming from $U_d$ is in fact explicit: it consists of the polynomials whose critical orbit has length smaller than $7$. Also, for all the polynomials outside of $U_d$, the proof gives a fairly explicit description of the possible finite exceptions $\alpha$. The proof of Theorem \ref{main: finiteness} combines the $1$-dimensionality principle with an appropriate "descent along the tree" argument.


We conclude by mentioning the surprising work of Looper \cite{looper}, that proves Conjecture \ref{conj} for a different family of polynomials. It uses completely different methods. The argument proceeds by ingeniously combining equidistribution theorems with the fact that congruences of inertia elements are \emph{simultaneously} holding for all places, in case the Galois group is abelian.

\section{The \texorpdfstring{$1$}{}-dimensionality principle}
The material of this section is divided in three parts.

In the first part, we develop the basic background of spherically homogenous trees, culminating in a group-theoretic version of the $1$-dimensionality principle.

In the second part, we give the basic preliminaries on the Galois groups of polynomial sequences and the associated trees. 

In the third part we combine the work of the previous two to establish that abelian images occur only in the postcritically finite case. We remark that it would be possible to extend this implication to slightly more general polynomial sequences, but we have here limited ourselves to the case of a single polynomial. 

Finally in the last part we develop more flexible $1$-dimensionality principles, which will be in later sections applied to situations where only a limited amount of roots of unity is available.  
\subsection{Spherically homogenous trees and their automorphism groups}
The setup of generalized trees has already appeared in the work of the first author \cite{ferra}. The main novelty of this section is the $1$-dimensionality principle. To arrive there in a natural way, we begin with a systematic presentation of the trees appearing in this work, fixing along the way some important notations. 

Let $\{n_k\}_{k \in \mathbb{Z}_{\geq 1}}$ be a sequence of integers valued in $\mathbb{Z}_{\geq 2}$. Fix once and for all an alphabet of symbols $\{x_i\}_{i \in \mathbb{Z}_{\geq 1}}$. Denote by $M_{\infty}$ the free monoid on these symbols, and by $e$ the identity elements of this monoid. We now define the set
$$T_{\infty}(\{n_k\}_{k \in \mathbb{Z}_{\geq 1}}),
$$
consisting of words $w$ in $M_{\infty}$ having the following shape. Either $w=e$. Or
$$w=x_{f(\ell)} \ldots x_{f(1)},
$$
with the constraint that $f(j)$ is in $\{1,\ldots,n_j\}$, for each $j \leq \ell$. We call \emph{length} of the word $w$, the number $\ell\eqqcolon\ell(w)$, with $\ell(e)=0$ by convention. We talk about the subset of \emph{level} $\ell$, to denote the finite collection $L(\ell)$ of words in $T_{\infty}(\{n_k\}_{k \in \mathbb{Z}_{\geq 1}})$ having length $\ell$: there are precisely $n_{\ell} \dots n_1$ such elements. 

We view $T_{\infty}(\{n_k\}_{k \in \mathbb{Z}_{\geq 0}})$ as a \emph{directed Cayley graph}, with respect to multiplication on the left by the generators $\{x_i\}_{i \in \mathbb{Z}_{\geq 1}}$. In this way, $T_{\infty}(\{n_k\}_{k \in \mathbb{Z}_{\geq 0}})$ becomes naturally a directed spherically homogeneous tree rooted at $e$: the outward degree of a vertex $w$ equals precisely $n_{\ell(w)+1}$. The spherical degree of such tree is precisely the sequence $\{n_k\}_{k\ge 1}$.

This situation is already sufficient to capture the Galois theory of \emph{general} sequences of polynomials. However this paper is devoted to sequences of \emph{unicritical} polynomials. It is for this reason that we need to put the following additional piece of structure on our trees. 

Let $w$ be a word in the tree with $\ell\coloneqq \ell(w)>0$. Let us write
$$w=x_iw',
$$
with necessarily $i \in \{1,\ldots,n_\ell\}$. We then place an oriented edge connecting $w$ with
$$x_{j}w',
$$
where $j$ is the unique number in $\{1,\ldots,n_\ell\}$ that is congruent to $i+1$ modulo $n_{\ell}$. To stress the presence of this additional labels we denote by
$$\widetilde{T}_{\infty}(\{n_k\}_{k \in \mathbb{Z}_{\geq 1}}),
$$
the resulting directed graph. Finally, given a non-negative integer $\ell$, we define
$$\widetilde{T}_{\ell}(\{n_k\}_{k \in \mathbb{Z}_{\geq 0}})=\bigcup_{0 \leq j \leq \ell}L_j,
$$
with the natural structure of subgraph (see Figure \ref{fig1} for an example).

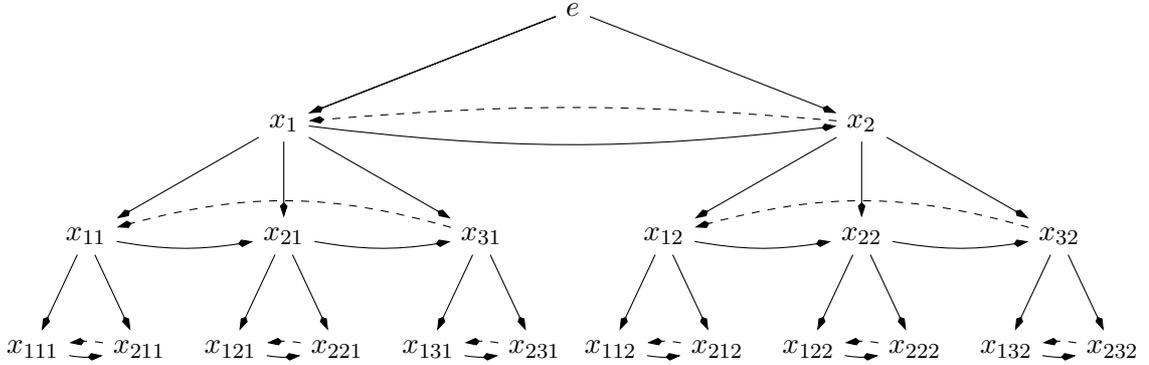
\begin{figure}[h]
\centering
\begin{tikzpicture}
\tikzstyle{solid node}=[circle,draw,inner sep=1.5,fill=black]
\node[]{$e$}[sibling distance=76mm]
child{node(x1)[]{$x_1$}[sibling distance=26mm]  edge from parent[->,-{Kite[]}]
child{node(x1x1)[]{$x_{11}$}[sibling distance=14mm]  edge from parent[->,-{Kite[]}]
child{node(x1x1x1){$x_{111}$}   edge from parent[->,-{Kite[]}]}
child{node(x2x1x1)[]{$x_{211}$}   edge from parent[->,-{Kite[]}]}}
child{node(x2x1)[]{$x_{21}$}[sibling distance=14mm]   edge from parent[->,-{Kite[]}]
child{node(x1x2x1)[]{$x_{121}$}   edge from parent[->,-{Kite[]}]}
child{node(x2x2x1)[]{$x_{221}$}   edge from parent[->,-{Kite[]}]}}
child{node(x3x1)[]{$x_{31}$}}[sibling distance=14mm]    edge from parent[->,-{Kite[]}]
child{node(x1x3x1)[]{$x_{131}$}  edge from parent[->,-{Kite[]}]}
child{node(x2x3x1)[]{$x_{231}$}   edge from parent[->,-{Kite[]}]}}
child{node(x2)[]{$x_2$} [sibling distance=26mm]   edge from parent[->,-{Kite[]}]
child{node(x1x2)[]{$x_{12}$}[sibling distance=14mm]   edge from parent[->,-{Kite[]}]
child{node(x1x1x2)[]{$x_{112}$}   edge from parent[->,-{Kite[]}]}
child{node(x2x1x2)[]{$x_{212}$}   edge from parent[->,-{Kite[]}]}}
child{node(x2x2)[]{$x_{22}$}[sibling distance=14mm]  edge from parent[->,-{Kite[]}]
child{node(x1x2x2)[]{$x_{122}$}   edge from parent[->,-{Kite[]}]}
child{node(x2x2x2)[]{$x_{222}$}   edge from parent[->,-{Kite[]}]}}
child{node(x3x2)[]{$x_{32}$}[sibling distance=14mm]  edge from parent[->,-{Kite[]}]
child{node(x1x3x2)[]{$x_{132}$}  edge from parent[->,-{Kite[]}]}
child{node(x2x3x2)[]{$x_{232}$}  edge from parent[->,-{Kite[]}]}}
};\draw[->,-{Kite[]}](x1) to [bend right=7] (x2);
 \draw[dashed,->,-{Kite[]}] (x2) to [bend right=5] (x1);
 \draw[->,-{Kite[]}] (x1x1) to [bend right=10] (x2x1);
 \draw[->,-{Kite[]}] (x2x1) to [bend right=10] (x3x1);
 \draw[dashed,->,-{Kite[]}] (x3x1) to [bend right=16] (x1x1);
 \draw[->,-{Kite[]}] (x1x2) to [bend right=10] (x2x2);
 \draw[->,-{Kite[]}] (x2x2) to [bend right=10] (x3x2);
 \draw[dashed,->,-{Kite[]}] (x3x2) to [bend right=16] (x1x2);
 \draw[->,-{Kite[]}] (x1x1x1) to [bend right=10] (x2x1x1);
 \draw[dashed,->,-{Kite[]}] (x2x1x1) to [bend right=10] (x1x1x1);
 \draw[->,-{Kite[]}] (x1x2x1) to [bend right=10] (x2x2x1);
 \draw[dashed,->,-{Kite[]}] (x2x2x1) to [bend right=10] (x1x2x1);
 \draw[->,-{Kite[]}] (x1x3x1) to [bend right=10] (x2x3x1);
 \draw[dashed,->,-{Kite[]}] (x2x3x1) to [bend right=10] (x1x3x1);
 \draw[->,-{Kite[]}] (x1x1x2) to [bend right=10] (x2x1x2);
 \draw[dashed,->,-{Kite[]}] (x2x1x2) to [bend right=10] (x1x1x2);
 \draw[->,-{Kite[]}] (x1x2x2) to [bend right=10] (x2x2x2);
 \draw[dashed,->,-{Kite[]}] (x2x2x2) to [bend right=10] (x1x2x2);
 \draw[->,-{Kite[]}] (x1x3x2) to [bend right=10] (x2x3x2);
 \draw[dashed,->,-{Kite[]}] (x2x3x2) to [bend right=10] (x1x3x2);
\end{tikzpicture}\caption{The graph $\widetilde{T}_3(\{2,3,2,\ldots\})$. We shortened $x_ix_j$ to $x_{ij}$ and $x_ix_jx_k$ to $x_{ijk}$ for graphical reasons.} \label{fig1}
\end{figure}

We now put
$$\Omega_{\ell}(\{n_k\}_{k \in \mathbb{Z}_{\geq 1}})\coloneqq\text{Aut}_{\text{graph}}(\widetilde{T}_{\ell}(\{n_k\}_{k \in \mathbb{Z}_{\geq 1}})),
$$
for each $\ell$ in $\mathbb{Z}_{\geq 1} \cup \{\infty\}$. Clearly $\Omega_{\ell}(\{n_k\}_{k \in \mathbb{Z}_{\geq 1}})$ preserves the words of every length up to $\ell$, in particular it preserves $L_\ell$, which is then naturally a $\Omega_{\ell}(\{n_k\}_{k \in \mathbb{Z}_{\geq 1}})$-set. 

As we next explain we naturally have that
$$\Omega_{\infty}(\{n_k\}_{k \in \mathbb{Z}_{\geq 1}})=\varprojlim_{\ell \in \mathbb{Z}_{\geq 0}}\Omega_{\ell}(\widetilde{T}_{\infty}(\{n_k\}_{k \in \mathbb{Z}_{\geq 1}})),
$$
and that
\begin{equation}\label{semi_direct}
\Omega_{\ell+1}(\{n_k\}_{k \in \mathbb{Z}_{\geq 1}})=(\mathbb{Z}/n_{\ell+1}\mathbb{Z})[L_\ell] \rtimes \Omega_{\ell}(\{n_k\}_{k \in \mathbb{Z}_{\geq 1}}),
\end{equation}

realizing $\Omega_{\infty}(\{n_k\}_{k \in \mathbb{Z}_{\geq 1}})$ both as a profinite group and as an infinite semi-direct product. Here $(\mathbb{Z}/n_{\ell+1}\mathbb{Z})[L_{\ell}]$ is viewed as an $\Omega_{\ell}(\{n_k\}_{k \in \mathbb{Z}_{\geq 1}})$-permutation module. 

We will now proceed constructing a joint realization of this semi-direct product construction, using the following \emph{standard topological generators} of the profinite group $\Omega_{\infty}(\{n_k\}_{k \in \mathbb{Z}_{\geq 1}})$.

Let $w$ be a word in $\widetilde{T}_\infty(\{n_k\}_{k\in \Z_{\ge 1}})$. We now define an automorphism of the graph 
$$\sigma_w,
$$
of order precisely equal to $n_{\ell(w)}+1$. Let first
$$w'=w^{''}x_jw,
$$
then we define 
$$\sigma_w(w')=w^{''}x_iw,
$$
where $i$ is the unique integer of $\{1,\ldots,n_{\ell(w)}\}$ congruent to $j+1$ modulo $n_{\ell(w)}$. For all the other words $w'$, we let $\sigma_w$ act as the identity. It is immediately clear that every automorphism in $\Omega_\ell(\{n_k\}_{k \in \mathbb{Z}_{\geq 1}})$ can be written in a unique way as a product
$$\prod_{k=0}^{\ell-1}\prod_{w\in L_k}\sigma_{w}^{e_w},$$
where $e_w\in \Z/n_{\ell(w)}\Z$. Looking at the exponents, this yields immediately a bijection:
$$(\phi_w)_{w \in \widetilde{T}_{\ell}(\{n_k\}_{k \in \mathbb{Z}_{\geq 1}})}\colon \Omega_{\ell}(\{n_k\}_{k \in \mathbb{Z}_{\geq 1}}) \to \prod_{w \in \widetilde{T}_{\ell-1}(\{n_k\}_{k \in \mathbb{Z}_{\geq 1}})} \mathbb{Z}/n_{\ell(w)}\mathbb{Z}, 
$$
and in turn, passing to the inverse limit, a bijection of profinite sets:
$$(\phi_w)_{w \in \widetilde{T}_{\infty}(\{n_k\}_{k \in \mathbb{Z}_{\geq 1}})}\colon \Omega_{\infty}(\{n_k\}_{k \in \mathbb{Z}_{\geq 1}}) \to \prod_{w \in \widetilde{T}_{\infty}(\{n_k\}_{k \in \mathbb{Z}_{\geq 1}})} \mathbb{Z}/n_{\ell(w)}\mathbb{Z}, 
$$
which we call \emph{portrait map} by a very natural extension of the standard portrait concept explained for example in \cite{nek}. The image of $\sigma$ under such map will be referred to as the \emph{portrait of $\sigma$}.

These considerations also show that that the group $\Omega_{\infty}(\{n_k\}_{k \in \mathbb{Z}_{\geq 1}})$ is naturally equal to $\varprojlim_{\ell \in \mathbb{Z}_{\geq 1}}\Omega_{\ell}(\widetilde{T}_{\infty}(\{n_k\}_{k \in \mathbb{Z}_{\geq 1}}))$ and explain the description \eqref{semi_direct}.

Under the isomorphism $\Omega_{\ell+1}(\{n_k\}_{k\in\Z_{\ge 1}})\to (\Z/n_{\ell+1}\Z)[L_\ell]\rtimes\Omega_\ell(\{n_k\}_{k\in \Z_{\ge 1}})$ an automorphism $\sigma\in \Omega_{\ell+1}(\{n_k\}_{k \in \mathbb{Z}_{\geq 1}})$ is mapped to the pair 
$$\left(\sum_{w\in L_{\ell}}\phi_w(\sigma)w,\overline{\sigma}\right),$$
where $\overline{\sigma}$ is the projection of $\sigma$ on $\Omega_{\ell}(\{n_k\}_{k \in \mathbb{Z}_{\geq 1}})$.

The portrait map is, of course, not a group homomorphism. However the iterated semi-direct product structure gives us relations of quasi-multiplicativity, using the following elementary lemma.
\begin{lemma}\label{cocycle_lemma}
Let $G_1,G_2$ be groups and let $A$ be an abelian $G_2$-module. Then a map $\tilde{\pi}:G_1 \to A \rtimes G_2$ is a surjective group homomorphism if and only if there exists a pair $(\psi,\pi)$ with $\pi\colon G_1\to G_2$ a surjective group homomorphism and $\psi\colon G_1\to A$ a $1$-cocycle whose restriction to $\ker \pi$ is surjective, such that $\tilde{\pi}(g)=(\psi(g),\pi(g))\in A\times G_2$ for every $g\in G_1$. 
\end{lemma}

Lemma \ref{cocycle_lemma}, applied to map \eqref{semi_direct} post-composed with the natural projection 
$$\Omega_{\infty}(\{n_k\}_{k \in \mathbb{Z}_{\geq 1}})\to \Omega_{\ell}(\{n_k\}_{k \in \mathbb{Z}_{\geq 1}}),$$
allows to conclude that for every $\ell\ge 1$ the map
$$\psi_\ell:\Omega_{\infty}(\{n_k\}_{k \in \mathbb{Z}_{\geq 1}}) \to (\mathbb{Z}/n_{\ell}\mathbb{Z})[L_{\ell-1}],
$$
$$\sigma\mapsto \sum_{w\in L_{\ell-1}}\phi_w(\sigma)w$$
is a $1$-cocycle. One can turn a $1$-cocycle into a homomorphism by making the action trivial, namely taking the co-invariants of the module, which in a permutation module corresponding to a transitive set, corresponds precisely to the trace map
$$\phi_\ell\coloneqq\sum_{w\in L_{\ell-1}}\phi_w.
$$
It is a well-known calculation that in this way one describes the abelianization of a semi-direct product: in other words
$$(\phi_\ell)_{\ell \in \mathbb{Z}_{\geq 1}}:\Omega_{\infty}(\{n_k\}_{k \in \mathbb{Z}_{\geq 1}})^{\text{ab}} \to \prod_{\ell \in \mathbb{Z}_{\geq 1}}\mathbb{Z}/n_\ell\mathbb{Z},
$$
is an isomorphism. Another way to view the construction of the map $\phi_\ell$ is as the map induced from the unique map of $\Omega_\ell(\{n_k\}_{k \in \mathbb{Z}_{\geq 1}})$-sets
\begin{equation}\label{abelianization_from_natural_map}
    L_{\ell-1} \twoheadrightarrow L_0,
\end{equation}

where the right hand side is the set with only one point. 

Next consider, for $\ell\geq 2$ an integer, the natural map
$$L_{\ell-1} \twoheadrightarrow L_1,
$$
of $\Omega_\ell(\{n_k\}_{k \in \mathbb{Z}_{\geq 1}})$-sets defined by picking the last to the right variable in a word: $x_{i_\ell} \ldots x_{i_1}$ is mapped into $x_{i_1}$. This induces a corresponding map $(\Z/n_\ell\Z)[L_{\ell-1}]\to (\Z/n_\ell\Z)[L_1]$ of $\Omega_\ell(\{n_k\}_{k \in \mathbb{Z}_{\geq 1}})$-modules. Pre-composing with $\psi_\ell$, we find a $1$-cocycle
$$\widetilde{\phi}_\ell:\Omega_{\infty}(\{n_k\}_{k \in \mathbb{Z}_{\geq 1}}) \to (\mathbb{Z}/n_\ell\mathbb{Z})[L_1],
$$
which is going to play a key role in the proof of the proposition below. We point out that a simple formula for $\widetilde{\phi}_\ell$ is given by:
$$
    \widetilde{\phi}_\ell(\sigma)=\sum_{i \in n_1}\left(\sum_{w \in L_{\ell-2}}\phi_{wx_i}(\sigma)\right)\cdot x_i.
$$
Observe that the sum makes sense because we are taking $\ell \geq 2$. 

\begin{proposition} \label{commutator pairing}
Let $\sigma,\tau$ be two elements of $\Omega_{\infty}(\{n_k\}_{k \in \mathbb{Z}_{\geq 1}})$ with 
$$[\sigma,\tau]=\textup{id}.
$$
Suppose that $\phi_1(\tau)=0$. Let $d_\sigma$ be the order of $\phi_1(\sigma)$ in $\Z/n_1\Z$. 

Then, for each $i$ in $\mathbb{Z}_{\geq 1}$, we have that $\phi_i(\tau)$ is in the group generated by $d_\sigma$ in $\mathbb{Z}/n_i\mathbb{Z}$. 
\begin{proof}
We can clearly assume that $i \geq 2$, otherwise the conclusion is directly implied by the assumption that $\phi_1(\tau)=0$, which certainly makes it divisible by $d_\sigma$. 

Let $\psi_i$ be the $1$-cocycle 
$$\psi_i\colon\Omega_{\infty}(\{n_k\}_{k \in \mathbb{Z}_{\geq 1}}) \to (\mathbb{Z}/n_i\mathbb{Z})[L_{i-1}],
$$
defined above. Since $\sigma\tau=\tau\sigma$, we have:
$$\sigma\psi_i(\tau)+\psi_i(\sigma)=\psi(\sigma\tau)=\psi(\tau\sigma)=\tau\psi_i(\sigma)+\psi_i(\tau)
$$
that can be re-organized as:
\begin{equation}\label{commutativity}
    (\sigma-1)\psi_i(\tau)=(\tau-1)\psi_i(\sigma).
\end{equation}
As explained above, we now apply the surjection of $\Omega_{\infty}(\{n_k\}_{k \in \mathbb{Z}_{\geq 1}})$-sets $L_{i-1} \twoheadrightarrow L_1$ (using again the fact that $i \geq 2$), yielding an epimorphism of $\Omega_{\infty}(\{n_k\}_{k \in \mathbb{Z}_{\geq 1}})$-modules
$$(\mathbb{Z}/n_i\mathbb{Z})[L_{i-1}] \twoheadrightarrow (\mathbb{Z}/n_i\mathbb{Z})[L_1],
$$
which postcomposed with $\psi_l$, yields a $1$-cocycle $\widetilde{\phi}_i$ as explained above. From \eqref{commutativity} we obtain therefore: 
$$(\sigma-1)\widetilde{\phi}_i(\tau)=(\tau-1)\widetilde{\phi}_i(\sigma).
$$
However, by assumption, $\tau$ acts trivially on $L_1$. Therefore the right hand side now vanishes. We conclude that the element
$$\widetilde{\phi}_i(\tau) \in (\mathbb{Z}/n_i\mathbb{Z})[L_1]
$$
is $\sigma$-invariant. However the orbits of $\sigma$ on $L_1$ are $n_1/d_\sigma$ cycles $\mathcal O_1,\ldots,\mathcal O_{n_1/d_\sigma}$ each of length $d_\sigma$. It follows that
$$\widetilde{\phi}_i(\tau)=\sum_{j=1}^{n_1/d_\sigma}c_j\sum_{w\in \mathcal O_j}w$$
where $c_j\in \Z/n_i\Z$ for every $j$. It follows that further postcomposing with
$$L_1 \twoheadrightarrow L_0
$$
yields at the same time $\phi_i(\tau)$ (as explained by \eqref{abelianization_from_natural_map}) and a multiple of $d_\sigma$, since $|\mathcal O_j|=d_\sigma$ for every $j$. This immediately implies the claim. 
\end{proof}
\end{proposition}
We can now give the main consequence of this fact, which is the main result of this subsection: the group-theoretic incarnation of the $1$-dimensionality principle. 
\begin{theorem} \label{thm: 1-dimensionality principle}
Let $G$ be an abelian subgroup of  $\Omega_{\infty}(\{n_k\}_{k \in \mathbb{Z}_{\geq 1}})$. Let $d_1$ be the size of $\phi_1(G)$. Then the image of the map
$$(\phi_1,((\phi_i)_{\textup{mod} \ d_1})_{i \geq 2}): G \to \mathbb{Z}/d_1\mathbb{Z} \times \prod_{i \geq 2} \mathbb{Z}/(d_1,n_i)\mathbb{Z},
$$
is a cyclic group of order $d_1$. 

In other words, the group generated by the characters
$$({(\phi_i)}_{\bmod d_1})_{i \geq 2},
$$
is cyclic generated by $\phi_1$. 
\begin{proof}
Let us call $A$ the named image and let
$$\pi\colon \mathbb{Z}/d_1\mathbb{Z} \times \prod_{i \geq 2} \mathbb{Z}/(d_1,n_i)\mathbb{Z}\to \Z/d_1\Z$$
be the projection on the first coordinate. Proposition \ref{commutator pairing} implies immediately that $A \cap \ker(\pi)=\{0\}$. Hence $\pi$ is injective when restricted to $A$, but on the other hand it is surjective as well, by hypothesis. Hence it must be an isomorphism, and the claim follows.  
\end{proof}
\begin{remark}
We chose the name "$1$-dimensionality principle", because of the variant of Theorem \ref{thm: 1-dimensionality principle} that we explain in the paragraph below. In this case, the conclusion is that a certain $\F_p$-vector space is $1$-dimensional, and hence the name.
\end{remark}
\end{theorem}

\subsubsection{Additional \texorpdfstring{$1$}{}-dimensionality principles} \label{limited}
The group-theoretic material built so far will suffice to handle polynomial sequences $a_kx^{n_k}+b_k$, as long as the base field contains $n_k$-th roots of unity. In our main result for quadratic number fields we will however need a useful $1$-dimensionality principles also when only \emph{some} of the $n_k$'s have this assumption. This is the motivation of the following addition to the previous section. 

In order to keep the exposition light, we will focus on the case that the sequence $\{n_k\}_{k \geq 1}$ is a sequence of primes, which already gives the additional $1$-dimensionality principles sufficient for our application. To stress this switch, we will denote $n_k\coloneqq p_k$. The prime $p_1\eqqcolon p$, will have a special status. 

Let us denote by $I\coloneqq\{i \in \mathbb{Z}_{\geq 1}\colon p_i=p\}$. The following variant of the graphs, presented in the previous section, will be applied in studying unicritical dynamical systems over fields possessing $p$-th root of unity, but not guaranteed to have a $p_k$-th root of unity when $k$ is not in $I$. As a set of vertices the graph is still
$$T_{\infty}(\{p_k\}_{k \in \mathbb{Z}_{\geq 1}}).
$$
Furthermore for each $i$ in $I$, we place $p$-cycles on $L_{k}$ exactly as we did in the previous section. Done this, we do not add any other edge. We denote by
$$\Omega_N(\{p_k\}_{k \in \mathbb{Z}_{\geq 1}},p),
$$
the automorphism group of the resulting graph up to level $N$, for each $N$ in $\mathbb{Z}_{\geq 1} \cup \{\infty\}$. Notice that this might differ from the group $\Omega_\infty(\{p_k\}_{k\in \Z_{\ge 1}})$ previously defined, since as soon as the sequence $\{p_k\}$ contains a prime $p_M$ different from $2$ and from $p$, the group $\Omega_M(\{p_k\}_{k\in \Z_{\ge 1}},p)$ is strictly larger than $\Omega_M(\{p_k\}_{k\in \Z_{\ge 1}})$, due to the lack of the extra horizontal edges at the $M$-th level of the corresponding graph. Reasoning precisely as in the previous section, we find semi-direct product decomposition for each $N$ in $I$, as follows
$$\Omega_{N}(\{p_k\}_{k \in \mathbb{Z}_{\geq 1}},p)=\mathbb{F}_p[L_{N-1}] \rtimes \Omega_{N-1}(\{p_k\}_{k \in \mathbb{Z}_{\geq 1}},p),
$$
yielding cocycles $\psi_n$ and characters $\phi_n$, precisely in the same way as before. With exactly the same argument, we arrive at the following $1$-dimensionality principle, justifying the name of the principle itself.
\begin{theorem} \label{1-dimensional principle for primes}
Let $G$ be an abelian subgroup of $\Omega_{\infty}(\{p_k\}_{k \in \mathbb{Z}_{\geq 1}},p)$. Suppose that $\phi_1(G)=\mathbb{F}_p$. Then the image of
$$(\phi_i)_{i \in I}:G \to \mathbb{F}_p^{I},
$$
is $1$-dimensional. 
\end{theorem}

\subsection{Galois groups of polynomial sequences}
The framework that we described in the previous subsection arises in the following situation, already introduced in \cite{ferra}. Let $\{n_k\}_{k\ge 1}$ a $\mathbb{Z}_{\geq 2}$-valued sequence of integers, let $K$ be a field and for $k\ge 1$ let $f_k=a_kx^{n_k}+b_k\in K[x]$, with $a_k\ne 0$.  From now on, for every $n\ge 1$ we let $f^{(n)}\coloneqq f_1\circ\ldots\circ f_n$ (and $f^{(0)}\coloneqq x$ by convention). We denote by $\mathcal F$ the sequence $\{f_k\}_{k\ge 1}$. 

\begin{definition}
Let $\alpha\in K$. The \emph{post-critical orbit} of the pair $(\FF,\alpha)$ is the sequence $\{f^{(n)}(0)-\alpha\}_{n\ge 1}$.

The \emph{adjusted post-critical orbit} of the pair $(\FF,\alpha)$ is the sequence $\{c_{n,\alpha}(\FF)\}_{n\ge 1}$ given by:
$$c_{n,\alpha}(\FF)\coloneqq\begin{cases}-f_1(0)+\alpha & \mbox{if $n_1=2$ and $n=1$}\\
						f^{(n)}(0)-\alpha & \mbox{otherwise}\end{cases}.$$
						
If the post-critical orbit of $(\FF,0)$ is finite, we say that $\FF$ is \emph{post-critically finite}.	

Furthermore we put
$$\widetilde{c}_{n,\alpha}(\FF):=(-1)^{\#L_{n-1}}\left(\frac{f^{(n)}(0)-\alpha}{\prod_{i=1}^{n}a_i^{\#L_{n-1}}}\right)
$$
\end{definition}
We now recall the following fundamental lemma.

\begin{lemma}[{{\cite[Lemma 2.6]{jones1}}}]\label{jones_lemma}
Let $f,g\in K[x]$, where $f,g$ have, respectively, degrees $d_f$ and $d_g$ and leading coefficients $a$ and $b$, and let $\gamma_1,\ldots,\gamma_{d_f-1}$ be the critical points of $f$. Set $\Delta_n\coloneqq \text{disc}(g\circ f^n)$. Then for every $n\ge 1$ we have:
$$\Delta_n=\pm\Delta_{n-1}^{d_f}d_f^{k_1}a^{k_2}b^{k_3}\prod_{i=1}^{d_f-1}g(f^{n-1}(\gamma)).$$
\end{lemma}
Lemma \ref{jones_lemma} immediately implies the following corollary.
\begin{corollary}\label{disc_corollary}
Let $\{n_k\}_{k\ge 1}$ be a $\mathbb{Z}_{\geq 2}$-valued sequence, and let $f=a_kx^{n_k}+b_k\in K[x]$ for every $k\ge 1$. Suppose that the characteristic of $K$ is coprime to every $n_k$. Let $\FF\coloneqq\{f_k\}_{k\ge 1}$ and $\alpha\in K$. Then $f^{(n)}-\alpha$ is separable for every $n$ if and only if $c_{n,\alpha}(\FF)\ne 0$ for every $n$.
\end{corollary}

From now on, we fix for every positive integer $n$ a generator $\zeta_n$ of the group of $n$-th roots of unity in $K^{\text{sep}}$. We now fix a sequence $\{n_k\}_{k\ge 1}$ in $\mathbb{Z}_{\geq 2}$, and a sequence of polynomials $f=a_kx^{n_k}+b_k\in K[x]$ with $a_k\ne 0$ for every $k\ge 0$ and $\alpha\in K$; we assume that the characteristic of $K$ is coprime to every $n_k$ and that $c_{n,\alpha}\ne 0$ for every $n$. Assume that for each $k \ge 1$, the element $\zeta_{n_k}$ is in $K$. 

Under these assumptions we can construct a rooted, infinite, spherically homogeneous tree of spherical degree $\{n_k\}_{k\ge 1}$, denoted by $T_\infty(\mathcal F,\alpha)$, whose root is $\alpha$ and whose nodes at distance $n$ from $\alpha$ are the roots of $f^{(n)}-\alpha$. Next, we add horizontal directed edges in the following way:  two roots $\alpha,\alpha'$ at the same level $\ell$  are connected with a directed edge from $\alpha$ to $\alpha'$ if and only if
$$\zeta_{n_\ell}\alpha=\alpha'.
$$
This yields a directed graph $\widetilde{T}_\infty(\mathcal F,\alpha)$ that is isomorphic to $\widetilde{T}_{\infty}(\{n_k\}_{k \in \mathbb{Z}_{\geq 1}})$. We therefore fix an identification
$$\iota\colon\widetilde{T}_\infty(\mathcal F,\alpha)\stackrel{\sim}{\to} \widetilde{T}_{\infty}(\{n_k\}_{k \in \mathbb{Z}_{\geq 1}}).$$

Thanks to our assumption that the polynomials have coefficients in $K$ and the relevant roots of unity  are also in $K$, we see that the absolute Galois group $G_K\coloneqq \gal(K^{\sep}/K)$ acts on $T_\infty(\mathcal F,\alpha)$ via automorphisms, giving rise to an associated arboreal representation 
$$\rho_{\mathcal F,\alpha}\colon G_K\to \Omega_\infty(\{n_k\}_{k\in \Z_{\ge 1}}).$$
The image of such representation can be identified with $G_\infty(\FF,\alpha)$. Hence for every $k\in \mathbb{Z}_{\geq 0}$ the composition of $\rho_{\mathcal F,\alpha}$ with $\phi_k|_{G_\infty(\mathcal F,\alpha)}$ yields a cyclic character of order $n_k$ of $G_K$, where $\phi_k$ is the homomorphism given in the previous section. By Kummer theory, the fixed field of the kernel of this character is obtained by adjoining to $K$ the $n_k$-th root of an element of $K$. The next lemma, which is the second key ingredient for proving the $1$-dimensionality principle, gives a formula for such an element in terms of the dynamics of $\FF$.


For an element $t\in K$ and $n$ positive integer, we denote by $\chi_t(n)\colon G_K\to \mathbb{Z}/n\mathbb{Z}$ the unique character that satisfies $\sigma(r)=\zeta_n^{\chi_t(\sigma)}r$ for every $\sigma\in G_K$ and every $r\in K^{\sep}$ such that $r^n=t$. Recall that for each $k$ in $\mathbb{Z}_{\geq 1}$, we have constructed a cocycle
$$\psi_k:\Omega(\{n_j\}_{j \in \mathbb{Z}_{\geq 1}}) \to (\mathbb{Z}/n_{k}\mathbb{Z})[L_{k-1}].
$$
This cocycle composed with the augmentation map gives the character $\phi_k$. 

\begin{lemma}\label{character_description}
Under the above assumptions, for every $k \in \mathbb{Z}_{\geq 1}$ we have
$$\phi_k|_{G_\infty(\FF,\alpha)}\circ \rho_{\FF,\alpha}=\chi_{\widetilde{c}_{k,\alpha}(\FF)}(n_{k}).$$
\end{lemma}
\begin{proof}
By an abuse of notation, in this proof words $w$ will refer to the nodes in $\widetilde{T}_\infty(\FF,\alpha)$ corresponding to $w$ under the identification $\iota$.

Let us pick a decomposition of $L_k$ into orbits $\mathcal{O}$ under the $G_K$-action. Then we have a decomposition of permutation modules
$$(\mathbb{Z}/n_k\mathbb{Z})[L_{k-1}]=\bigoplus_{\mathcal{O} \ \text{orbit}} (\mathbb{Z}/n_k\mathbb{Z})[\mathcal{O}].
$$
Let us fix an orbit $\mathcal{O}$. Then the component of $\psi_k$ in the above decomposition is simply
$$\psi_{\mathcal{O}}(\sigma)=\sum_{w \in \mathcal{O}}\phi_w(\sigma)\cdot w.
$$
Let $w$ be a word in $\mathcal{O}$. Let $\mathcal{H}_w$ be the stabilizer of $w$ in $G_K$. We see that $(\mathbb{Z}/n_k\mathbb{Z})[\mathcal{O}]$ is the induction of the trivial $\mathcal{H}_w$-module $\mathbb{Z}/n_k\mathbb{Z}$ to $G_K$. Hence we can apply Shapiro's Lemma (see for example \cite[Proposition 6.2]{Brown}) and obtain an isomorphism
\begin{equation}\label{shapiro}
    \Hom(\mathcal H_w,\Z/n_k\Z)\stackrel{\sim}{\to} H^1(G_K,(\Z/n_k\Z)[\mathcal O])
\end{equation}
that we can compose with the map
$$H^1(G_K(\Z/n_k\Z)[\mathcal O])\to \Hom(G_K,(\Z/n_k\Z))$$ induced by the augmentation map $(\mathbb{Z}/n_k\mathbb{Z})[\mathcal{O}]\to \mathbb{Z}/n_k\mathbb{Z}
$.
The argument at \cite[p.\ 83]{Brown} shows that the composition
\begin{equation}\label{corestriction}
    \Hom(\mathcal H_w,\Z/n_k\Z)\to \Hom(G_K,(\Z/n_k\Z))
\end{equation}
is just the corestriction map.

Now consider the element ${\phi_w}_{|\mathcal{H}_w}\in \Hom(\mathcal H_w,\Z/n_k\Z)$; under isomorphism \eqref{shapiro} this becomes precisely $\psi_{\mathcal O}$. Hence choosing a set of representatives $W$ for the orbits of $L_k$ under the $G_K$-action and summing up all the corresponding corestriction maps \eqref{corestriction}, we get a map
$$\Theta\colon \bigoplus_{w\in W}\Hom(\mathcal H_w,\Z/n_k\Z)\to \Hom(G_K,(\Z/n_k\Z)).$$

Since $\psi_k$ is simply the collection of all $\psi_{\mathcal O}$'s, it follows that
\begin{equation}\label{global_augmentation}
    \Theta(({\phi_w}_{|\mathcal{H}_w})_{w\in W})=\phi_k.
\end{equation}

On the other hand, $\mathcal H_w=\gal(K^{sep}/K(w))$, so that after the canonical Kummer identifications, map \eqref{corestriction} becomes the norm map
$$\frac{K(w)^{\times}}{K(w)^{\times n_k}} \to \frac{K^{\times}}{K^{\times n_k}}.
$$
Under the Kummer identification $\Hom(\mathcal H_w,\Z/n_k\Z)\to K(w)^{\times}/{K(w)^{\times}}^{n_k}$, each element ${\phi_w}_{|\mathcal{H}_w}$ becomes $(w-b_k)/a_k$. Therefore, \eqref{global_augmentation} shows that $\phi_k$, viewed as an element of $\frac{K^{\times}}{K^{\times n_k}}$, is nothing else than the class of
$$\prod_{w\colon \ell(w)=k-1} \frac{w-b_k}{a_k} \in \frac{K^{\times}}{K^{\times n_k}},
$$
which can be rewritten as
$$(-1)^{\#L_{k-1}}\prod_{w\colon \ell(w)=k-1} \frac{w-b_k}{a_k}=(-1)^{\#L_{k-1}}\cdot \frac{f^{(k-1)}(b_k)-\alpha}{\text{leading coefficient of} f^{(k)}}=
$$
$$=(-1)^{\#L_{k-1}}\left(\frac{f^{(k)}(0)-\alpha}{\prod_{i=1}^{k}a_i^{\#L_{k-1}}}\right)=\widetilde{c}_{k,\alpha}.
$$
\end{proof}

Lemma \ref{character_description} immediately implies the following corollary.

\begin{corollary}\label{kernel_subgroup}
Let $d$ be a positive integer. Let $(b_i)_{i\in \mathbb{Z}_{\geq 1}}\in \bigoplus_{i\in \mathbb{Z}_{\geq 1}}\mathbb{Z}/d\mathbb{Z}$ be a vector supported at nonnegative integers $i$ such that $d|n_i$. Then $G_\infty(\FF,\alpha)\subseteq \ker (\sum_{i\in \mathbb{Z}_{\geq 1}}b_i\phi_i)$ if and only if $\left(\prod_{i \geq 1}\widetilde{c}_{i,\alpha}\right)^{b_i}\in {K^{\times}}^{d}$.
\end{corollary}

We are now ready to state and prove an arithmetic incarnation of the $1$-dimensionality principle.

\begin{theorem}\label{1dimensionality_general}
Let $\{n_k\}_{k\ge 1}$ be a sequence in $\mathbb{Z}_{\geq 2}$. Let $K$ be a field such that $K$ contains a primitive $n_k$-th root of unity for every $k\ge 0$. For every $k\ge 1$, let $f_k=a_kx^{n_k}+b_k\in K[x]$ with $a_k\ne 0$, let $\FF\coloneqq \{f_k\}_{k\ge 1}$ and let $\alpha\in K$. Suppose that $c_{n,\alpha}\ne 0$ for every $n\ge 1$ and let $d_1$ be the degree of the field defined by $f_1-\alpha$. Suppose that $G_\infty(\FF,\alpha)$ is abelian.

Then the set $\left\{(\widetilde{c}_{k,\alpha}(\FF))^{\frac{d_1}{(d_1,n_k)}}\colon k\in \mathbb{Z}_{\geq 2} \right\}$ spans a cyclic subgroup of the abelian group $K^{\times}/{K^{\times}}^{d_1}$ generated by the element $\gamma_1$, which is an element of $K^{\times}$ such that 
$$\gamma_1^{\frac{n_1}{d_1}}=\widetilde{c}_{1,\alpha}(\FF).
$$
\end{theorem}
\begin{proof}
This is an immediate translation of Theorem \ref{thm: 1-dimensionality principle} by means of Lemma \ref{character_description}. 
\end{proof}
\subsection{Abelian implies PCF}
Let us now turn to the case where $K$ is in addition also a \emph{number field}. The following theorem shows that a polynomial gives an \emph{abelian} dynamical Galois group only if it is PCF. 



\begin{theorem}\label{abelian_implies_pcf}
Let $K$ be a number field and $f=ax^d+b\in K[x]$ with $d\ge 2$ and $a\ne 0$. Suppose that $G_\infty(f,\alpha)$ is abelian, then $f$ is PCF.
\end{theorem}
We remark that if on the one hand Theorem \ref{abelian_implies_pcf} follows also from \cite[Theorem 1.2]{ferra2}, which proves it with a completely different argument and for more general rational functions. On the other hand the $1$-dimensionality principle at the basis of the present proof is still a valid tool to rule out abelian images \emph{also} when $f$ is PCF. 

Before proving the theorem, let us prove a proposition that will be helpful both in the proof and in later sections. 
\begin{proposition} \label{exceptional}
 Let $K$ be a number field. Let $f=ax^d+b$, with $a \neq 0$ and let $\alpha$ be in $K$. Then the following are equivalent: 
 \begin{enumerate}
     \item The group $G_{\infty}(f,\alpha)$ is finite.
     \item The group $G_{\infty}(f,\alpha)$ is trivial.
     \item The tree $T_{\infty}(f,\alpha)$ is finite.
     \item The tree $T_\infty(f,\alpha)$ is contained in the critical orbit of $f$.
     \item We have that $b=0$ and $\alpha=0$.
 \end{enumerate}
\begin{proof}
 The implications $(5)\implies (4)$, $(3)\implies (2)$ and $(2)\implies (1)$ are all trivial. Since the post-critical orbit of $K$ is all contained in $K$, the implication $(4)\implies (3)$ follows immediately from the Northcott property, together with the fact that the set $T_{\infty}(f,\alpha)$ is of bounded height. The same argument shows that $(1)\implies (3)$. So we are only left with proving that $(3)\implies (5)$. Suppose that we are in $(3)$, then the entire tree $T_{\infty}(f,\alpha)$ must be contained in the critical orbit of $f$: suppose indeed a node $\beta$ is not, then $f^{-n}(\beta)$ has exactly $d^n$-distinct elements by Lemma \ref{jones_lemma}, which would give a contradiction with $(3)$ for $n$ sufficiently large. On the other hand every point, except at most one (that is in case the orbit is finite) point $x$ where the orbit starts a cycle, in the critical orbit has at most one preimage in the critical orbit. This shows that $T_{\infty}(f,\alpha)-\{x\}$ must be at most the point $b$, since $f$ is unicritical, and $b$ is the unique point having only $1$ pre-image. So by now we already know that $T_{\infty}(f,\alpha)$ consists of at most $\{x,b\}$, and in case the critical orbit is infinite we even already know that $T_{\infty}(f,\alpha)$ is $\{b\}$, so its size in particular is at most $1$. If it is exactly $1$, then it means that it must be constantly equal to $b$. But if it contains $b$, then it has to contain also $0$, and since it has size $1$, it must be $b=\alpha=0$. The other case is that it equals $\{x,b\}$. Again if we want to avoid $b=0$, it must be $x=0$. And at the same time it needs to have at least $2$ pre-images, which therefore have to be both $0$ and $b$. But $f(0)=b\neq 0$, so we reach a contradiction. This gives the desired conclusion. 
\end{proof}
\end{proposition}
Let us now prove Theorem \ref{abelian_implies_pcf}.
\begin{proof}[Proof of Theorem \ref{abelian_implies_pcf}]
Thanks to Proposition \ref{exceptional}, we can readily assume that $T_{\infty}(f,\alpha)$ is infinite. If not, we certainly have that $f=ax^d$ is PCF. We can also restrict the representation to $K(\zeta_d)$, so there is no loss of generality in assuming that $\zeta_d$ is in $K$. Furthermore, up to passing to a further finite extension we can find a node of the tree that is outside of the critical orbit. So there is no loss of generality in assuming that the entire tree is outside of the critical orbit. 

Let us first observe that the set
$$T_{\infty}(f,\alpha) \cap K
$$
is finite. Indeed $T_{\infty}(f,\alpha)$ is an infinite set of bounded height and thus the claim follows from the Northcott property. 

Let us denote by $\text{Max}(f,\alpha)$ the set of nodes that are in $K$ and have the entire subtree below them outside of $K$. Notice that $\text{Max}(f,\alpha)$ is not empty, since the descendants of a given node are either all in $K$ or all outside $K$ thanks to our assumption that $\zeta_d\in K$. Now for each $\beta$ in $\text{Max}(f,\alpha)$, we are in place to apply Theorem \ref{1dimensionality_general}. We find a divisor $d_1|d$ strictly larger than $1$ and a list of elements $\{\gamma_i\}_{i \in \mathbb{Z}/d_1\mathbb{Z}}$ such that the following happens. For each $n \geq 2$ we can find $i$ in $\mathbb{Z}/d_1\mathbb{Z}$ and $y_n$ in $K^{\times}$ such that
$$y_n^{d_1}=\gamma_i \cdot (f^{n}(0)-\beta).$$
Thanks to piegonhole, if the critical orbit were infinite, we would find one $\gamma_{i_0}$ such that the curve
$$C\colon y^{d_1}=\gamma_{i_0} \cdot (f^{3}(x)-\beta)
$$
has infinitely many $K$-rational points. This contradicts Faltings Theorem, because $f^{3}(x)-\beta$ has degree at least $6$ and it is separable, since $\beta$ is outside of the critical orbit. Hence the genus of $C$ is strictly bigger than $1$. To avoid this contradiction, the critical orbit must be finite, which is the desired conclusion. 
    
\end{proof}
\subsection{The \texorpdfstring{$1$}{}-dimensionality principle for limited roots of unity}
In the interest of later arithmetical applications, we want to be able to drop the assumption that for \emph{each} $k$, the field $K$ contains a primitive $n_k$-th root of unity, and still obtain $1$-dimensionality principles. In this auxiliary section, we carry this over in the very special case where the $n_k$'s are all primes. The proofs are a straightforward adaptation of the ones from the previous section, hence we omit them in the interest of brevity.

From now on, we fix a sequence of primes $\{p_k\}_{k\ge 1}$, polynomials $f_k=a_kx^{p_k}+b_k\in K[x]$ with $a_k\ne 0$ for every $k\ge 1$ and $\alpha\in K$; we assume that the characteristic of $K$ is different from every $p_k$, we let $\FF\coloneqq \{f_k\}_{k\ge 1}$ and we assume that $c_{n,\alpha}(\FF)\ne 0$ for every $n$. Finally, we let $p\coloneqq p_1$ and
$$I\coloneqq \{n\ge 1\colon p_n=p\},$$
and we assume that $K$ contains a primitive $p$-th root of unity $\zeta_p$.

Notice that the absolute Galois group $G_K\coloneqq \gal(K^{\sep}/K)$ acts on $T_\infty(\mathcal F,\alpha)$ via automorphisms, giving rise to an associated arboreal representation $\rho_{\mathcal F,\alpha}$. The image of such representation can be identified with $G_\infty(\FF,\alpha)$. Hence for every $n\in I$ the composition of $\rho_{\mathcal F,\alpha}$ with $\phi_n|_{G_\infty(\mathcal F,\alpha)}$ yields a cyclic character of order $p$ of $G_K$, where $\phi_n$ is the homomorphism given in Section \ref{limited}. By Kummer theory, the fixed field of the kernel of this character is obtained by adjoining to $K$ the $p$-th root of an element of $K$.

In fact putting together Theorem \ref{1-dimensional principle for primes} and a straightforward adaptation of Lemma \ref{character_description}, we obtain the following. 

\begin{theorem}\label{1dimensionality}
Let $\{p_k\}_{k\ge 1}$ be a sequence of primes, and let $p\coloneqq p_1$. Let $K$ be a field such that $\text{char }K\ne p_k$ for every $k\ge 1$ and containing a primitive $p$-th root of unity. Let $I\coloneqq \{k\in \N \colon p_k=p\}$. For every $k\ge 1$, let $f_k=a_kx^{p_k}+b_k\in K[x]$ with $a_k\ne 0$, let $\FF\coloneqq \{f_k\}_{k\ge 1}$ and let $\alpha\in K$. Suppose that $c_{n,\alpha}\ne 0$ for every $n\ge 1$ and that $f_1-\alpha$ is irreducible. If $G_\infty(\FF,\alpha)$ is abelian, then the set $\left\{\frac{c_{n,\alpha}(\FF)}{a_1}\colon n\in I\right\}$ spans a $1$-dimensional subspace of the $\F_p$-vector space $K^{\times}/{K^{\times}}^p$.
\end{theorem}

Let us explain how Theorem \ref{1dimensionality} can be also applied to the study of the dynamical Galois group of a single unicritical polynomial. Let $d$ be a positive integer and $p$ a prime dividing $d$. Given $f=ax^d+b\in K[x]$ with $a\ne 0$, we associate a polynomial sequence to $f$ as follows. Write $d=p^nq_1\cdot \ldots \cdot q_r$, where $n\ge 1$ and the $q_i$'s are primes all different from $p$ (but not necessarily pairwise distinct). Let $t\coloneqq n+r$ (notice that $r$ can be $0$). Then we can associate to $f$ a sequence $\FF_f(p)$ defined by $\{f_k\}_{k\ge 1}$ where $f_1\coloneqq ax^p+b$, $f_i\coloneqq x^p$ for every $i\in \{2,\ldots,n\}$, $f_j=x^{q_j}$ for every $j=1,\ldots,r$ and $f_{k+t}=g_k$ for every $k\ge 1$. Notice that then $f^{(nt)}=f^n$ for every $n\ge 1$.

\begin{definition}
We call $\FF_f(p)$ the \emph{$p$-sequence} attached to $f$.
The sequence
$$c^p_{n,\alpha}(f)=\begin{cases} -f(0)+\alpha & \mbox{ if $n=1$ and $p=2$}\\
f^n(0)-\alpha & \mbox{otherwise}
\end{cases}$$
is called \emph{$p$-adjusted post-critical orbit} of the pair $(f,\alpha)$.
\end{definition}

\begin{remark}\label{orbit_rmk}
Notice that the $p$-adjusted post-critical orbit of $(f,\alpha)$ is a subsequence of the adjusted post-critical orbit of the pair $(\FF_f(p),\alpha)$. Moreover, elements of $\{c_{n,\alpha}(\FF_f(p))\}_{n\ge 1}$ that do not belong to $\{c^p_{n,\alpha}(f)\}_{n\ge 1}$ are just powers of elements of the latter sequence. Therefore, $c^p_{n,\alpha}(f)\ne 0$ for every $n\ge 1$ if and only if $c_{n,\alpha}(\FF_f(p))\ne 0$ for every $n\ge 1$.
\end{remark}

\begin{remark}\label{moving_bt_trees}
Given $f=ax^d+b$ and $\alpha\in K$ such that $c_{n,\alpha}(f)\ne 0$ for every $n$, one can of course construct the usual infinite, rooted, $d$-regular tree $T_\infty(f,\alpha)$ associated to $f$, as explained in \cite{jones2}. This can be obtained starting from the tree $T_\infty(\FF_f(p),\alpha)$ in the following way: if $d=pq_1\ldots q_n$, where the $q_i$'s are primes (and $n=0$ by convention if $d=p$), then for every $t\ge 1$ nodes at distance $t(n+1)$ from the root $\alpha$ in $T_\infty(\FF_f(p),\alpha)$ coincide with nodes at distance $t$ from the root $\alpha$ in $T_\infty(f,\alpha)$. Moreover, two nodes at distance $t(n+1)$ from the root in $T_\infty(\FF_f(p),\alpha)$ have the same ancestor at level $(t-1)(n+1)$ if and only if they have the same ancestor at level $t-1$ when they are considered as nodes at level $t$ of $T_\infty(f,\alpha)$. Moreover, the groups $G_\infty(f,\alpha)$ and $G_\infty(\FF_f(p),\alpha)$ are topologically isomorphic, although they act on non-isomorphic trees as long as $d\ne p$. Hence one can equivalently study abelianity of either of the two groups.
\end{remark}

\begin{corollary}\label{working_corollary}
Let $d$ be a positive integer, $p$ a prime dividing $d$ and $K$ be a field of characteristic different from all primes dividing $d$ and containing a primitive $p$-th root of unity. Let $f\coloneqq ax^d+b\in K[x]$ with $a\ne 0$. Let $\FF_f(p)=\{f_k\}_{k\ge 1}$ be the $p$-sequence attached to $f$. Let $\alpha\in K$ be such that $(\alpha-b)/a\notin (-1)^{p-1}K^p$ and suppose that $c^p_{i,\alpha}(f)\ne 0$ for every $i\ge 1$. Let 
$$S\coloneqq \bigcup_{j=1}^nj+t\N.$$
If $G_\infty(f,\alpha)$ is abelian, then $\left\{\frac{c_{s,\alpha}(\FF_f)}{a}\colon s \in S\right\}$ spans a $1$-dimensional subspace of $K^{\times}/{K^{\times}}^p$. In particular, the sequence $\left\{\frac{c^p_{n,\alpha}(f)}{a}\right\}_{n\ge 1}$ spans a $1$-dimensional subspace of $K^{\times}/{K^{\times}}^p$.
\end{corollary}


\begin{remark}\label{working_remark}
Notice how in Corollary \ref{working_corollary}, $p$ can be any prime dividing $d$. This means that if $f=ax^d+b$, $\mathcal P$ is the set of primes dividing $d$, $K$ contains a primitive $p$-th root of unity for every $p\in \mathcal P$, and $(\alpha-b)/a$ is not a $p$-th power for every $p\in \mathcal P$, then the set $\left\{\frac{c_{i,\alpha}(f)}{a}\right\}_{i\ge 1}$ spans a $1$-dimensional subspace of $K^{\times}/{K^{\times}}^p$ for every $p\in \mathcal P$.
\end{remark}

\section{Polynomials with periodic critical point}\label{Section: periodic}
This section is devoted to show that unicritical polynomials with periodic critical point never admit abelian arboreal representations, unless they fall into the case predicted by Conjecture \ref{conj}.

Let $K$ be a number field, $d$ be a positive integer and $a,c,\alpha\in K$ with $a\ne 0$ so that the polynomial 
$$f(x)\coloneqq ax^d-c,
$$
has the property that $f^{n}(0)=0$ for some $n$ in $\mathbb{Z}_{\geq 1}$. We will simply say that $0$ is $f$-periodic and we will call the smallest positive $n$ for which the relation holds the \emph{period of} $0$. The goal of this section is to classify all pairs 
$$(f(x),\alpha),
$$
with $f(x)$ as above and with
$$G_{\infty}(f,\alpha)\coloneqq \text{Gal}(K_{\infty}(f,\alpha)/K),
$$
\emph{abelian}, where
$$K_{\infty}(f,\alpha)\coloneqq\bigcup_{n \in \mathbb{Z}_{\geq 0}}K(f^{-n}(\alpha)) \subseteq K^{\text{sep}},
$$
with $K^{\text{sep}}$ a fixed algebraic closure of $K$.
\begin{theorem} \label{classify periodic f}
Let $K$ be a number field and $d \in \mathbb{Z}_{\geq 2}$. Let $a,c$ be in $K$ with $a\ne 0$ such that $0$ is $f$-periodic, where $f\coloneqq ax^d-c$. Suppose there exists $\alpha \in K$ such that $G_{\infty}(f,\alpha)$ is abelian. Then $c=0$ and $(f,\alpha)$ is $K^{\text{ab}}$-conjugate to $(x^d,\gamma)$, where $\gamma$ is either $0$ or a root of unity. In particular, if $a=1$ and $G_\infty(f,\alpha)$ is abelian then $c=0$ and $\alpha$ is either $0$ or a root of unity.
\end{theorem}

Let us begin with a key proposition. This is directly inspired from \cite[Lemma~1.2]{ahmad}.
\begin{proposition} \label{there are all the d-th roots of the tree} 
Let $K$ be a number field, $d \geq 2$ an integer, $c,\alpha$ in $K$ and assume that $x^d-c$ has periodic critical orbit, and $\alpha$ is not an element of this orbit. Suppose furthermore that $c \neq 0$. Then for all $\beta \in T_{\infty}(f,\alpha)$ we have that
$$\{\gamma \in K^{\textup{sep}}:\gamma^{d^n}=-\beta \mbox{ for some } n\in \N\} \subseteq K_{\infty}(f,\alpha). 
$$
\begin{proof}
Take $\beta$ in $T_{\infty}(f,\alpha)$. Denote by $n_0$ the period of $0$. By definition, we have that
$$f^{-n_0}(\beta) \subseteq T_{\infty}(f,\alpha).
$$
Observe that, since $n_0>0$, and since $f(x)$ has the form $x^d-c$ the set $f^{-n_0}(\beta)$ is stable under multiplication of $\mu_d(K^{\text{sep}})$, in other words is a union of $\mu_d(K^{\text{sep}})$-cosets. Denote by $\mathcal{C}$ the set of cosets, and pick for each coset $t \in \mathcal{C}$ a representative $\gamma_t$ in the coset. We now have that
$$(-1)^{\#\mathcal{C}}\left(\prod_{t \in \mathcal{C}}\gamma_t\right)^d=(-1)^{\#\mathcal{C}} \cdot \prod_{t \in \mathcal{C}}(f(\gamma_t)+c)=\prod_{\lambda \in f^{-n_0+1}(\beta)}(-c-\lambda)=$$
$$=f^{n_0-1}(-c)-\beta=f^{n_0}(0)-\beta=-\beta.
$$
Since $\alpha$ is not in the orbit of $0$, we have that $\#\mathcal{C}=d^{n_0-1}$ by Lemma \ref{jones_lemma}. Since $c \neq 0$, we have that $n_0-1>0$ and therefore we have shown that $-\beta$ admits a $d$th root in $K_{\infty}(f,\alpha)$, and since the latter extension is Galois, it contains the entire set $\{\gamma \in K^{\textup{sep}}:\gamma^d=-\beta\}$ as claimed. Iterating this construction, precisely in the way explained in \cite[Proposition 3.1]{pag} yields the desired conclusion. 
\end{proof}
\end{proposition}
We shall use the following simple fact about number fields.
\begin{proposition} \label{not too many dth roots}
Let $K$ be a number field. Let $a$ be in $K^{*}$. Suppose that
$$K(\{\gamma \in K^{\textup{sep}}: \gamma^{d^n}=a\}_{n \in \mathbb{Z}_{\geq 0}})/K,
$$
is abelian. Then $a$ is a root of unity. 
\begin{proof}
This fact can be proved in a completely algebraic way and for much more general fields $K$. However, for the sake of brevity, here we will argue simply by noticing that in case $a$ is not a root of unity then $\{\gamma \in K^{\textup{sep}}: \gamma^{d^n}=a\}_{n \in \mathbb{Z}_{\geq 0}}$ contains elements of arbitrarily small height, while thanks to Amoroso--Zannier's work \cite{AZ} there is a lower bound for the non-zero heights of non-zero elements of $K^{\text{ab}}/K$. This gives the desired conclusion. 
\end{proof}
\end{proposition}

To conclude the proof we will need the following auxiliary fact. 

\begin{proposition} \label{trees of roots of unity}
Let $K$ be a number field $d$ be in $\mathbb{Z}_{\geq 2}$, and $c, \alpha$ be in $K$, with $\alpha$ not in the critical orbit of $f$, where $f(x)=x^d-c$. Suppose that $T_{\infty}(f,\alpha)$ consists entirely of roots of unity. Then $c=0$ and $\alpha$ is a root of unity. 
\begin{proof}
Clearly since $\alpha$ is in $T_{\infty}(f,\alpha)$ the assumptions imply that $\alpha$ is a root of unity. Since $\alpha$ is not in the $f$-orbit of $0$ we know that the tree $T_{\infty}(f,\alpha)$ is infinite. Now fix any embedding $\sigma$ of $K^{\text{sep}}$ into $\mathbb{C}$. We have that the complex plane curve $C:=\{y=x^d-\sigma(c)\}$ has by assumption infinitely many solutions $(x,y) \in C(\mathbb{C})$ both roots of unity. Therefore we conclude by Lang's Theorem \cite{lang1} that $C(\mathbb{C})$ must contain the image of the complex curve $\{(\lambda_1t^{n_1},\lambda_2t^{n_2})\}_{t \in \mathbb{C}^{*}}$, for some $\lambda_1,\lambda_2 \in \mathbb{C}^{*}$ and $n_1,n_2$ positive integers. But this means that the polynomial
$$\lambda_2t^{n_2}-\lambda_1t^{dn_1}-\sigma(c),
$$
vanishes identically, which forces $c=0$. 
\end{proof}
\end{proposition}
\begin{remark}
In the proof of Proposition \ref{trees of roots of unity} we opted to invoke Lang's theorem only for the sake of brevity. Actually one of the proof's of Lang's theorem suggests self-contained argument for Proposition \ref{trees of roots of unity}. Namely one can prove that the roots of unity belonging to $T_{\infty}(f,\alpha)$ must be dense in the unit circle (after applying the embedding $\sigma$). Therefore $f$ must stabilize the unit circle, which evidently forces $c=0$, since otherwise we can pick $x_0$ on the unit circle such that $x_0^d$ is the projection of $-\sigma(c)$ on the unit circle, which gives
$$|x_0^d-\sigma(c)|_{\infty}=|\sigma(c)|_{\infty}+1>1,
$$
for $c \neq 0$.  
\end{remark} 

We are now in position to conclude the proof of the main result of this section.

\begin{proof}[Proof of Theorem \ref{classify periodic f}] Let $f=ax^d-c\in K[x]$ and $\alpha\in K$, and suppose that $K_{\infty}(f,\alpha)/K$ is an abelian extension. Let $L\coloneqq K(\sqrt[d-1]{a})$. Let now $m=ux\in L[x]$, where $u=1/\sqrt[d-1]{a}$, $f'\coloneqq m^{-1}\circ f\circ m$, so that $f'=x^d-c/u$, and $\alpha'=\alpha/u$. Clearly the extension $L_\infty(f',\alpha')/L$ is still abelian and $0$ is $f'$-periodic. If there is a point $\beta\in T_\infty(f',\alpha')$ that is not in the post-critical orbit of $f'$, then observing that $L(\beta)_\infty(f',\beta)/L(\beta)$ must be abelian as well, Proposition \ref{there are all the d-th roots of the tree} combined with Proposition \ref{not too many dth roots} forces $T_{\infty}(f',\beta)$ to consist entirely of roots of unity. Then we conclude with Proposition \ref{trees of roots of unity} that $c/u=0$, and hence $c=0$. Now \cite[Theorem 12]{andrews} proves that $(ax^d,\alpha)$ is $K^{\text{ab}}$-conjugate to $(x^d,\zeta)$ for some root of unity $\zeta$ (notice that since $\beta$ is not in the post-critical orbit of $f'$ then in particular it is not exceptional). Hence we are left only with the case where the entire tree $T_{\infty}(f',\alpha')$ is in the critical orbit of $f$. Proposition \ref{exceptional} shows that this can only happen if $\alpha'=c/u=0$.
\end{proof}

\section{PCF unicritical polynomials over quadratic number fields}

The goal of this section is to compute all conjugacy classes of PCF, monic, unicritical polynomials over quadratic number fields. We start with an elementary lemma.

\begin{lemma}\label{orbit_escape}
Let $K$ be a number field, $c\in K$ and $f=x^d+c\in K[x]$, where $d\ge 2$ is an integer. Let $\{c_n\}_{n\ge 1}$ be the post-critical orbit of $f$. If $f$ is PCF, then $c$ is an algebraic integer and $|c_n|_v\le 2$ for every archimedean place $v$ of $K$ and every $n\ge 1$.
\end{lemma}
\begin{proof}
If $v$ is a non-archimedean valuation of $K$ such that $v(c)<0$, then $v(c_n)=d^{n-1}v(c)$ for every $n\ge 1$, and therefore $f$ is not PCF. Hence $v$ is an algebraic integer.

Now $v$ be an archimedean place of $K$ and suppose that there is a smallest positive integer $N\in \N$ such that $|c_N|_v>2$. Then
$$|c_{N+1}|_v=|c_N^d+c|_v\ge |c_N|_v(|c_N|_v^{d-1}-|c|_v/|c_N|_v)>|c_N|_v,$$
where the last inequality follows from the fact that $|c|_v/|c_N|_v\le 1$ by construction. This shows by induction that the sequence $\{|c_{n+N}|_v\}_{n\ge 1}$ is strictly increasing, and therefore $f$ is not PCF.
\end{proof}
\begin{theorem}\label{PCF_classification}
Let $d\geq 2$ be an integer, $K$ be a quadratic number field and $f_c=x^d+c\in K[x]$. Then $f_c$ is PCF if and only if one of the following holds.
\begin{enumerate}
\item $f_c=x^d$;
\item $f_c=x^2-2$;
\item $d$ is even and $f_c=x^d-1$;
\item $d\equiv 3\bmod 4$ and $f_c=x^d\pm i$;
\item $d\equiv 4\bmod 6$ and $f_c=x^d+\zeta_6$, for $\zeta_6$ a primitive $6$th root of unity;
\item $f_c=x^2\pm i$;
\item $d\equiv 0\bmod 6$ and $f_c=x^d+\zeta_3$, for $\zeta_3$ a primitive third root of unity.
\end{enumerate}
\end{theorem}
\begin{proof}
One verifies easily that all polynomials listed in the statement are PCF, so the rest of the proof will be devoted to show that if $f_c$ is PCF, then it falls in one of those cases.

By Lemma \ref{orbit_escape}, if $f_c$ is PCF then $c$ is an algebraic integer whose logarithmic Weil height is bounded by $\log 2$; by Northcott theorem that there are only finitely many $c\in \overline{\Q}$ of degree $2$ over $\Q$ such that $h(c)\le \log 2$. A brief computation via Magma \cite{magma} shows that the set of all such $c$'s is (we only list one $c$ for each Galois conjugacy class):
\begin{align*}
S\coloneqq & \left\{  0, \pm 1,\pm 2,\sqrt{2},\sqrt{3},\pm\frac{1+\sqrt{5}}{2},i,2i,\pm(1+i),i\sqrt{2},\pm(1+i\sqrt{2}),\pm\frac{3+i\sqrt{3}}{2},\right.\\ 
 &    \left.\pm(1+i\sqrt{3}),\pm\frac{1+i\sqrt{3}}{2},i\sqrt{3},\pm\frac{1+i\sqrt{7}}{2},\pm\frac{3+i\sqrt{7}}{2},\pm\frac{1+i\sqrt{11}}{2},\pm\frac{1+i\sqrt{15}}{2}\right\}.
\end{align*}

Now it is just a matter of using Lemma \ref{orbit_escape} to rule out values of $c$ that do not give rise to PCF polynomials.

Let $c\in S\setminus\left\{0,\pm 1, i,\pm\frac{1+\sqrt{-3}}{2}\right\}$ and $d\geq 2$. By Lemma \ref{orbit_escape}, in order to show that $f_{c,d}=x^d+c$ is not PCF\footnote{We write $f_{c,d}$ at this point to underline the dependence on the two parameters $c,d$.} it is enough to find an archimedean place $v$ and a point $c_n=c_n(d)$ in the post-critical orbit of $f_{c,d}$ such that $|c_n|_v>2$. Since for all these $c$'s there is some $v$ with $|c|_v>1$, it is clear that there is some $d_0$, that depends on $c$, such that $|c_2(f_{c,{d_0}})|_v\ge |c|_v^{d_0}-|c|>2$ and of course for all $d\ge d_0$ one has $|c_2(f_{c,d})|_v>|c_2(f_{c,d_0})|_v$. This leaves out finitely many pairs $(c,d)$, which can be addressed one by one, and again a verification via Magma shows that, except for the case $c=-2$, $d=2$, the post-critical orbits always escape and therefore these $c$'s do not give rise to PCF polynomials for any $d$.

This leaves us with $c\in \left\{0,\pm 1, i,\pm\frac{1+\sqrt{-3}}{2}\right\}$. If $c=0$ then $f_{0,d}$ is PCF for every $d$, while $f_{1,d}$ is never PCF since $|c_3|=2^d>2$. When $c=-1$ and $d$ is odd we have $|c_3|=2^d-1>2$, while if $d$ is even then $f_{-1,d}$ is PCF.

To conclude the proof we need to treat the cases $c\in\left\{i,\pm\frac{1+\sqrt{-3}}{2}\right\}$. From now on we set $\zeta_6\coloneqq \frac{1+\sqrt{-3}}{2}$ and $\zeta_3\coloneqq \frac{-1+\sqrt{-3}}{2}$, and we will use $|\cdot |$ for the usual complex absolute value.

First, let $c=i$. Let $f_{i,d}=x^d+i$ with $d>2$ and $d\not\equiv 3 \bmod 4$. Then $|c_2|\in \{2,\sqrt{2}\}$ according to the residue class of $d$ modulo $4$; in any case $|c_2|\geq \sqrt{2}$ and hence $|c_3|=|c_2^d+i|\geq |c_2|^d-1\geq 2^{d/2}-1\ge 3$. It follows that $f_{i,d}$ is not PCF. 

Next, let $f_{\zeta_6,d}=x^d+\zeta_6$. An easy verification shows that if $d>2$ and $d\not\equiv 3,4\bmod 6$, then $|c_2|\geq \sqrt{3}$. Hence, $|c_3|\geq 3^{d/2}-1>2$. If $d=2$, then $|c_3|=\sqrt{7}>2$. When $d\equiv 3\bmod 4$ we have $c_2=-1+\zeta_6=\zeta_3$, so that $c_3=1+\zeta_6$ and in turn $|c_4|\geq 3^{d/2}-1>2$ again.

Finally, let $f_{\zeta_3,d}=x^d+\zeta_3$. If $d\equiv 1,4\bmod 6$ we have $c_2=2\zeta_3$, so that $|c_3|>2$. When $d\equiv 2\bmod 6$ we have $c_4=2\zeta_3$ and hence $|c_5|>2$. When $d\equiv 3,5\bmod 6$ we have $c_3=-1+\zeta_3$ so that $|c_3|=\sqrt{3}$ and $|c_4|\geq 3^{d/2}-1>2$.
\end{proof}

\begin{proposition}\label{case345}
Let $f$ be as in cases $(3),(4)$ or $(5)$ of Theorem \ref{PCF_classification}. Then $G_\infty(f,\alpha)$ is non-abelian for every $\alpha\in \overline{\Q}$.
\end{proposition}
\begin{proof}
This is a direct application of Theorem \ref{classify periodic f}.
\end{proof}
The cases $f_c=x^d,x^2-2$ were dealt with in \cite{andrews}; the case of $x^d$ is an immediate consequence of Amoroso--Zannier, as articulated for example already in Section \ref{Section: periodic}. The case $x^2-2$ can also be handled with Amoroso--Zannier \cite{AZ}, and it was originally explained already in \cite{andrews}.

The next two sections are devoted to rule out the remaining cases.

\section{The polynomial \texorpdfstring{$x^2+i$}{}}

Throughout the whole section we set $K\coloneqq \Q(i)$, $\O_K\coloneqq \Z[i]$, $f=x^2+i\in K[x]$ and $\O_{K,2}=\{z\in K\colon v(z)\geq 0 \mbox{ for every odd valuation $v$ of }K \}$.

The goal of this section is to prove the following theorem.

\begin{theorem}\label{case_6}
The group $G_\infty(f,\alpha)$ is non-abelian for every $\alpha\in K$.
\end{theorem}

In order to prove Theorem \ref{case_6}, we first  recall the following lemma, that was essentially proved in \cite[Section~3]{ferpag1}.

\begin{lemma}\label{2_integral}
Let $\alpha\in K$ and suppose that $G_\infty(f,\alpha)$ is abelian. Then $\alpha\in \O_{K,2}.$
\end{lemma}
\begin{proof}
Suppose there is an odd valuation $v$ of $K$ such that $v(\alpha)< 0$. Then it follows from \cite[Lemma~7.1]{anderson1} that $K_\infty(f,\alpha)/K$ is infinitely ramified at $v$. However, \cite[Lemma~3.11]{ferpag1} shows that this contradicts the abelianity of $G_\infty(f,\alpha)$.
\end{proof}

The following theorem is the key result of the section.

\begin{theorem}\label{non_abelian_gauss_integers}
Let $\alpha\in \O_{K,2}$ be such that $f-\alpha$ is irreducible, and let $G_3(f,\alpha)\coloneqq \gal(f^3-\alpha)$. Then $G_3(f,\alpha)$ is not abelian.
\end{theorem}
\begin{proof}
Suppose that $G_3(f,\alpha)$ is abelian. Notice that the irreducibility of $f-\alpha$ implies immediately that $\alpha$ does not belong to the adjusted post-critical orbit of $f$, i.e.\ $\alpha\notin\{\pm i,-1+i\}$. From now on, for every $\beta\in \overline{K}$ we will let, as usual:
$$c_{1,\beta}\coloneqq -i+\beta,\quad c_{2,\beta}\coloneqq-1+i-\beta,\quad c_{3,\beta}\coloneqq -i-\beta.$$
Since $G_3(f,\alpha)$ is abelian and $f-\alpha$ is irreducible it follows by Theorem \ref{1dimensionality} (or even \cite[Proposition~3.4]{ferpag1}) that $\langle c_{1,\alpha},c_{2,\alpha},c_{3,\alpha}\rangle\subseteq K^{\times}/{K^{\times}}^2$ is a 1-dimensional $\F_2$-vector space. Since we are assuming that $c_{1,\alpha}\notin K^2$, this leaves us with four possibilities:
\begin{enumerate}[(A)]
\item $c_{1,\alpha}c_{2,\alpha}\in K^2$ and $c_{3,\alpha}\in K^2$;
\item $c_{1,\alpha}c_{3,\alpha}\in K^2$ and $c_{2,\alpha}\in K^2$;
\item $c_{1,\alpha}c_{2,\alpha}\in K^2$ and $c_{1,\alpha}c_{3,\alpha}\in K^2$;
\item $c_{2,\alpha}\in K^2$ and $c_{3,\alpha}\in K^2$.
\end{enumerate}

We will now examine the four cases one by one.

(A) Write $c_{1,\alpha}c_{2,\alpha}=-\alpha^2+(2i-1)\alpha+1+i=v_0^2$ and $c_{3,\alpha}= -i-\alpha=u_0^2$ for some $u_0,v_0 \in K$. It follows that $(u_0,v_0)$ is a $K$-rational affine point on the following elliptic curve:
$$E_A\colon v^2=-u^4 + (-4i + 1)u^2 + 2i + 4.$$
One checks that $E_A(K)\cong \Z/4\Z$, and its only affine points are $(u,v)=(\pm(-i+1),0)$, which both give $\alpha=i$, and in turn a contradiction.

(B) Write $c_{1,\alpha}c_{3,\alpha}=-\alpha^2-1=v_0^2$ and $c_{2,\alpha}= -1+i-\alpha=u_0^2$ for some $u_0,v_0 \in K$. It follows that $(u_0,v_0)$ is a $K$-rational affine point on the following elliptic curve:
$$E_B\colon v^2=-u^4 + (2i - 2)u^2 + 2i - 1.$$
One checks that $E_B(K)\cong \Z/4\Z$, and its only affine points are $(u,v)=(\pm i,0)$, which both give $\alpha=i$, and in turn a contradiction.

(C) Write $c_{1,\alpha}c_{2,\alpha}=-\alpha^2+(2i-1)\alpha+1+i=u_0^2$ and $c_{1,\alpha}c_{3,\alpha}=-\alpha^2-1=v_0^2$ for some $u_0,v_0\in K$. Notice that since $\alpha\in \O_{K,2}$, then $u_0,v_0\in \O_{K,2}$. Hence $(u_0,v_0)$ is an $\O_{K,2}$-rational point on the curve:
\begin{equation}\label{genus0_1_curve}
C_1\colon u^4 - 2u^2v^2 - (2i + 4)u^2 + v^4 + (1-2i)v^2=0
\end{equation}
This is an irreducible curve of genus 0, but its function field has four places at infinity. Hence, it has finitely many $\O_{K,2}$-rational, thanks to a classic result of Siegel (see \cite[Theorem~5.1]{lang2} for a modern account), and the algorithm described in \cite{alvanos} can be used to find all of them; this is done in Theorem \ref{genus0_1}: it turns out that they are $\left\{(0,0),\left(\pm \frac{3+i}{8}, \pm \frac{7+9i}{8}\right)\right\}$. These values lead to $\alpha=i$, which contradicts our assumptions, or $\alpha=\frac{9i-7}{8}$; one verifies via Magma \cite{magma} that $G_3\left(f,\frac{9i-7}{8}\right)$ is a non-abelian group of order 32.

(D) This is the most involved case to handle. Write $c_{2,\alpha}= -1+i-\alpha=u_0^2$ and $c_{3,\alpha}= -i-\alpha=v_0^2$ for some $u_0,v_0\in K$. Notice that $u_0,v_0\in \O_{K,2}$. We set $d\coloneqq -i+\alpha$, so that
$$d=-1-u_0^2=-2i-v_0^2.$$
Since $d=c_{1,\alpha}$, it is not a square in $K$ by assumption. We will make use several times of the following elementary fact: if $w,z\in K^*$ are such that $w+z\sqrt{d}\in K(\sqrt{d})^2$, then the polynomial $4x^4-4wx^2+dz^2$ has a root in $K$.

First we claim that $f-\sqrt{d}$ is irreducible over $K(\sqrt{d})$, which is equivalent to say that $-i+\sqrt{d}\notin K(\sqrt{d})^2$. Suppose this is false. Then $4x^4+4ix^2+d$ has a root in $K$. Now recall that $d=-2i-v_0^2$. Then the curve
$$E_D\colon v^2=4x^4+4ix^2-2i$$
has an affine point. One checks that $E_D(K)\cong (\Z/2\Z)^2$, and its only affine points are $(x,v)=(0,\pm(1-i))$, yielding $\alpha=i$ and therefore a contradiction.

Now, since $G_3(f,\alpha)$ is abelian then in particular $G_3'\coloneqq\gal(f^2-\sqrt{d}/K(\sqrt{d}))$, that is one of its subquotients, is abelian. It is a well-known fact (see for example \cite{stoll}) that $G_3'\cong D_8$ if and only if the $\F_2$-vector space $\langle  c_{1,\sqrt{d}},c_{2,\sqrt{d}}\rangle\subseteq K^{\times}/{K^{\times}}^2$ is $2$-dimensional. Conversely, if this condition does not hold then $G_3'$ is a proper subgroup of $D_8$ and therefore it is abelian. Since we proved above that $f-\sqrt{d}$ is irreducible, i.e.\ that $c_{1,\sqrt{d}}\notin K(\sqrt{d})^2$, it follows that $G_3'$ is abelian if and only if one of the following two conditions hold:
\begin{enumerate}[(i)]
\item $c_{2,\sqrt{d}}\in K(\sqrt{d})^2$;
\item $c_{1,\sqrt{d}}c_{2,\sqrt{d}}\in K(\sqrt{d})^2$.
\end{enumerate}

(i) Since $c_{2,\sqrt{d}}=-1+i-\sqrt{d}\in K(\sqrt{d})^2$, the polynomial $4x^4-4(i-1)x^2+d$ has a root in $K$. Since $d=-1-u_0^2$, it follows that the curve
$$C_D'\colon u^2=4x^4-4(i-1)x^2-1$$
has an affine $K$-rational point. One checks that the Jacobian $E'_D$ of $C_D'$ is a rank 0 elliptic curve with $E_D'(K)\cong \Z/2\Z\times\Z/4\Z$, and the only affine $K$-rational points of $C'_D$ are: $(x,u)=(0, \pm i), (\pm 1/2(i + 1), \pm(i + 1))$. These give $\alpha\in\{i,-i-1\}$ and therefore they both yield to a contradiction: $\alpha=i$ is forbidden by our assumptions and one can verify that $G_3(f,-1-i)\cong C_2\times D_4$, which is not abelian.

(ii) Since $c_{1,\sqrt{d}}c_{2,\sqrt{d}}=i+1-d+(2i-1)\sqrt{d}\in K(\sqrt{d})^2$, the polynomial $4x^4-4(i+1-d)x^2-(3+4i)d$ has a root in $K$. Notice that since $d=-1-u_0^2$ and $u_0\in\O_{K,2}$ then such root lies in $\O_{K,2}$ as well. This means that the curve
\begin{equation}\label{genus0_2_curve}
C_2\colon  4x^4 -4u^2x^2 + (4i + 3)u^2 + (-4i - 8)x^2 + 4i + 3=0
\end{equation}
has an $\O_{K,2}$-rational point. The curve $C_2$ has genus 0 but its function field has 4 places at infinity, so by the same argument we used in point (C), $C_2$ has finitely many $\O_{K,2}$-rational points: by Theorem \ref{genus0_2} they are $\left\{(\pm i,0),\left(\pm \frac{7i - 5}{8}, \pm \frac{2-i}{2}\right)\right\}$. Now $u_0=\pm i$ leads to $\alpha=i$, contradicting our assumptions, while $u_0=\pm \frac{7i - 5}{8}$ leads to $\alpha=\frac{67i - 20}{32}$; one can check that $G_3\left(f,\frac{67i - 20}{32}\right)$ is a non-abelian group of order 16.
\end{proof}

We are now ready to prove Theorem \ref{case_6}.

\begin{proof}[Proof of Theorem \ref{case_6}]
Let $\alpha\in K$, and suppose that $G_\infty(f,\alpha)$ is abelian. By Lemma \ref{2_integral}, we must have $\alpha\in \O_{K,2}$. Now it is easy to observe that the tree $T_\infty(f,\alpha)$ must contain a node $\beta$ that belongs to $K$ and is such that $f-\beta$ is irreducible, for if this does not happen, then by Northcott theorem $\alpha$ must be exceptional for $f$, but $f$ does not admit any exceptional point. Now of course $\beta$ must be $2$-integral as well, since $f^n(\beta)=\alpha$ for some $n$, and hence we can apply Theorem \ref{non_abelian_gauss_integers}, that tells us that $G_\infty(f,\beta)$ is non-abelian. But this is a subgroup of $G_\infty(f,\alpha)$, and hence the latter is non-abelian, too. 
\end{proof}

\section{The polynomial \texorpdfstring{$x^{6k}+\zeta_3$}{}}\label{cyclotomic}

Throughout this section, we let $\zeta_3$ a primitive $3$rd root of unity, $L\coloneqq \Q(\zeta_3)$ and $f_{6k}\coloneqq x^{6k}+\zeta_3$ where $k\ge 1$ is an integer.

The goal of this section is to prove the following theorem.

\begin{theorem}\label{case_7}
For every $\alpha\in L$ and every integer $k\ge 1$, the group $G_\infty(f_{6k},\alpha)$ is non-abelian.
\end{theorem}

We start with a subsection containing some results that we will use later on.

\subsection{Preliminary Galois theoretic results.}

\begin{lemma}[{{\cite[Theorem 9.1]{lang3}}}]\label{lang_lemma}

Let $K$ be a field, $a\in K$ be non-zero, and $r\geq 2$ be an integer. Assume that for every prime $p\mid r$ we have $a\notin K^p$, and if $4\mid r$ assume that $a\notin -4K^4$. Then $x^r-a$ is irreducible in $K[x]$.
\end{lemma}

\begin{remark}\label{galois_embedding}

As noticed in \cite[p. 300]{lang3}, if $K$ is a field of characteristic $0$, $r\geq 2$ is an integer and $a\in K$ is non-zero, the Galois group of the polynomial $x^r-a$ over $K$ embeds in the group of $2\times 2$ matrices of the form $\left(\begin{array}{cc} 1 & 0\\ b & d\end{array}\right)$ where $b\in \Z/r\Z$ and $d\in (\Z/r\Z)^*$. From now on, this group will be denoted by $G_r$.

\end{remark}

\begin{theorem}[{{\cite[Theorem 9.4]{lang3}}}]\label{lang_theorem}
Let $K$ be a field of characteristic $0$ and $r\geq 2$ an integer such that $[K(\zeta_r):K]=\phi(r)$, where $\phi$ is Euler's totient function. Let $a\in K$ and suppose that for every prime $p\mid r$, $a$ is not a $p$-th power in $K$. Then the Galois group of $f$ is isomorphic to $G_r$.
\end{theorem}

\begin{lemma}\label{abelian_subgroups}
Let $p$ be a prime and $n\geq 1$ an integer. Let $H\leq G_{p^n}$ be an abelian subgroup. Then $|H|\leq p^n$.
\end{lemma}
\begin{proof}
Let us first assume that $p$ is odd.

As shown in \cite[p. 300--301]{lang3}, the commutator subgroup $[G_{p^n},G_{p^n}]$ consists of all matrices of the form $\left(\begin{array}{cc} 1 & 0\\ b & 1\end{array}\right)$ with $b\in \Z/p^n\Z$. This implies that the quotient group $G_{p^n}/[G_{p^n},G_{p^n}]$ is isomorphic to $(\Z/p^n\Z)^*$, and the quotient map $\pi_n$ takes a matrix to its lower-right entry. By the first isomorphism theorem, we have that $|H|=|H\cap [G_{p^n},G_{p^n}]|\cdot |\pi_n(H)|$. Given the shape of $[G_{p^n},G_{p^n}]$, the subgroup $H\cap [G_{p^n},G_{p^n}]$ is isomorphic to a subgroup of the form $p^m(\Z/p^n\Z)$ for some $0\leq m\leq n$. Let $\left(\begin{array}{cc} 1 & 0\\ b & 1\end{array}\right)\in H\cap [G_{p^n},G_{p^n}]$ be a generator. Any $\left(\begin{array}{cc} 1 & 0\\ a & d\end{array}\right)\in H$ must commute with it, and one readily verifies that this happens if and only if $d\equiv 1\bmod p^{n-m}$; this being true for every element of $H$ in turn this implies that $|\pi_n(H)|\le p^m$. Hence, we must have $H\le p^{n-m}\cdot p^m=p^n$.

For $p=2$ the proof follows the very same logic but requires extra observations. First, the commutator subgroup $[G_{2^n},G_{2^n}]$ consists of all matrices of the form $\left(\begin{array}{cc} 1 & 0\\ b & 1\end{array}\right)$ with $b\in 2(\Z/2^n\Z)$ and therefore the quotient $G_{2^n}/[G_{2^n},G_{2^n}]$ is isomorphic to $(\Z/2\Z)\times(\Z/2^n\Z)^*$. The quotient map $\pi_n\colon G_{2^n}\to G_{2^n}/[G_{2^n},G_{2^n}]$ takes  $\left(\begin{array}{cc} 1 & 0\\ b & d\end{array}\right)$ to the pair $(b \bmod 2,d)$. Again, if $\left(\begin{array}{cc} 1 & 0\\ b & 1\end{array}\right)$ is a generator of $H\cap [G_{2^n},G_{2^n}]\cong 2^m(\Z/2^n\Z)$ (notice that this time $1\leq m\leq n$) and $\left(\begin{array}{cc} 1 & 0\\ a & d\end{array}\right)\in H$, then $d\equiv 1 \bmod 2^{n-m}$. The same computation shows that if $\left(\begin{array}{cc} 1 & 0\\ b & 1\end{array}\right)\in H$ and $b\equiv 1\bmod 2$, then necessarily $H\cong \Z/2^n\Z$, and the claim holds. Thus we can assume that if $\left(\begin{array}{cc} 1 & 0\\ b & 1\end{array}\right)\in H$, then $b\equiv 0\bmod 2$. Next, notice that if $\left(\begin{array}{cc} 1 & 0\\ b & d\end{array}\right),\left(\begin{array}{cc} 1 & 0\\ b' & d\end{array}\right)\in H$ and $d\not\equiv 1\bmod 2^n$, then the fact that the two elements commute forces $b\equiv b'\bmod 2$. This proves that $|\pi_n(H)|=|\overline{\pi}\circ\pi_n(H)|$, where $\overline{\pi}\coloneqq (\Z/2\Z)\times(\Z/2^n\Z)^*\to (\Z/2^n\Z)^*$ is the projection onto the second factor, and we can now conclude as we did in the odd case.
\end{proof}

\begin{lemma}\label{abelianity_level1}

Let $\alpha\in L$, and let $p\in \{2,3\}$. Let $n\geq 1$ and $f=x^{p^n}-\alpha\in L[x]$. If $\gal(f)$ is abelian then $\alpha=\zeta \cdot \beta^{p^{n-1}}$ for some $\zeta,\beta\in L$ where $\zeta$ is a root of unity.

\end{lemma}
\begin{proof}
If $\alpha$ is a root of unity or $n=1$ there is nothing to prove, so assume that this is not the case. Then there exists a largest $m\geq 0$ such that $\alpha=\zeta \cdot \beta^{p^m}$ for some $\zeta,\beta\in L$ with $\zeta$ a root of unity. If $m\geq n$ there is nothing to prove. Hence we can assume that $m<n$.

Now set $\widetilde{L}\coloneqq L$ if $p=3$ and $\widetilde{L}\coloneqq L(i)$ if $p=2$. Notice that in any case $\widetilde{L}\subseteq L(\zeta_{p^{n+1}})$. Then $\beta$ is not a $p$-th power in $\widetilde{L}$: if $p=3$ this holds by the assumption on $m$; if $p=2$ and $\beta$ were a square in $\widetilde{L}$, we would have $\beta=(a+bi)^2=a^2-b^2+2abi$ for some $a,b\in L$, and hence either $b=0$, implying that $\beta\in L^2$ which is impossible by the assumption on $m$, or $a=0$ so that $\beta=-b^2$ and consequently $\alpha=\pm\zeta\cdot b^{2^{m+1}}$, which is again impossible by the assumption on $m$.

Now notice that $f$ factors as $\prod_{j=1}^m(x^{p^{n-m}}-\zeta'^j\cdot \beta)$ over $L(\zeta_{p^n})$, where $\zeta'\in L(\zeta_{p^n})$ is a root of unity. We claim that for every $j$, the term $\zeta'^j\cdot \beta$ is not a $p$-th power in $L(\zeta_{p^n})$. In fact, suppose by contradiction that $\zeta'^j\cdot \beta=y^p$ for some $y\in L(\zeta_{p^n})$. Then $\beta$ is a $p$-th power in $L(\zeta_{p^{n+1}})$, and hence it is so in $\widetilde{L}(\zeta_{p^{n+1}})$. However, as we proved above, $\beta$ is not a $p$-th power in $\widetilde{L}$; it follows that the extension $\widetilde{L}(\sqrt[p]{\beta})/\widetilde{L}$ is cyclic of degree $p$. On the other hand, such extension must be contained in $L(\zeta_{p^{n+1}})/\widetilde{L}$, which is a cyclic $p$-extension, and therefore it contains a unique cyclic subextension of degree $p$. Therefore, by Kummer theory, we have that if $p=3$ then $\beta\cdot\zeta_3=\gamma^3$ for some $\gamma\in L$, implying that $\alpha=\zeta\gamma^{3^{m+1}}$ and contradicting the assumption; when $p=2$ then $\beta\cdot i=\gamma^2$ for some $\gamma\in \widetilde{L}$, and taking norms from $\widetilde{L}$ to $L$ this implies that $-\beta^2$ is a square in $L$, which never holds true as long as $\beta\ne 0$.

We have thus proved that $\zeta'^j\cdot \beta$ is not a $p$-th power in $L(\zeta_{p^n})$ and this implies, by Lemma \ref{lang_lemma},\footnote{Notice that since $n\geq 2$ then $-1$ is a square in $L(\zeta_{2^n})$ and therefore if $c\in L(\zeta_{2^n})$ is not a square, then in particular it is also not of the form $-4d^4$.} that $x^{p^{n-m}}-\zeta'^j\cdot \beta$ is irreducible over $L(\zeta_{p^n})$.

To end the proof, notice that the splitting field $M$ of $f$ contains $L(\zeta_{p^n})$, and therefore $[M:L]=[M:L(\zeta_{p^n})][L(\zeta_{p^n}):L]=p^{n-1}[M:L(\zeta_{p^n})]$. But what we proved above implies that $[M:L(\zeta_{p^n})]\geq p^{n-m}$ and hence $[M:L]\geq p^{2n-m-1}$. On the other hand, by Remark \ref{galois_embedding}, $\gal(M/L)$ embeds in $G_{p^n}$. Hence if $\gal(M/L)$ is abelian then by Lemma \ref{abelian_subgroups} its order is at most $p^n$. Therefore $p^{2n-m-1}\leq p^n$, which in turn implies that $m\geq n-1$.
\end{proof}

Now we shall start working our way towards the proof of Theorem \ref{case_7}. We start by ruling out basepoints $\alpha$ such that $\alpha-\zeta_3$ is $0$ or a root of unity.

\subsection{\texorpdfstring{$\alpha-\zeta_3$}{} is \texorpdfstring{$0$}{} or a root of unity.}

The key idea for ruling out these cases is the following proposition, that essentially follows from \cite[Theorem 1.2]{AZ}.

\begin{proposition} \label{AD for radicals}
Let $\gamma$ be an element of $L^{\textup{ab}}$ with $0<h(\gamma) \leq 1$. Suppose to have a positive integer $k$ and an element $\beta$ of $L^{\textup{ab}}$ such that $\beta^{6k}=\gamma$ and such that the extension $L(\beta)/L$ ramifies only at primes of $L$ whose residue characteristic is at most $6k$. 

Then we must have $k \leq 36$. 

\begin{proof}
We can assume that $k \geq 2$ otherwise we are immediately done. It follows that $6k>7$ and thus we are in position to apply the result of \cite{breusch} and conclude that there exists a prime $p$ congruent to $1$ modulo $3$ that sits between $6k$ and $12k$. Observe that $p$ then splits $L$. Let us fix one of the two primes above $p$ in $L$, and denote it by $\mathfrak{p}$. By our assumption on the ramified primes of $L(\beta)/L$ having all residue characteristic at most $6k$, we know that $L(\beta)/L$ is unramified at $\mathfrak{p}$. Hence, since we also have that $h(\beta)=\frac{1}{6k}h(\gamma)>0$, we have that
$$h(\beta) \geq \frac{\text{log}\left(\frac{\sqrt{p}}{2}\right)}{p+1},
$$
owing to \cite[p. 493]{AZ}, which we can apply because we have both assumptions: $h(\beta)>0$ and $\mathfrak{p}$ is unramified in $L(\beta)/L$.

At the same time, the assumption that $1 \geq h(\gamma)=6k\cdot h(\beta)$, gives us that
$$\frac{1}{6k} \geq h(\beta) \geq \frac{\text{log}\left(\frac{\sqrt{p}}{2}\right)}{p+1}.
$$
Hence, overall, we find that
$$\frac{1}{6k} \geq \frac{\text{log}\left(\frac{\sqrt{6k}}{2}\right)}{12k+1},
$$
which gives
$$k \leq \frac{2}{3}\text{exp}\left(4+\frac{1}{3k}\right).
$$
One verifies easily that for this inequality to hold, it must be that $k\le 36$.
\end{proof}
\end{proposition}

\begin{theorem}\label{az_trick}
Let $\alpha\in L$ be such that $\alpha-\zeta_3$ is $0$ or a root of unity. Then $G_\infty(f_{6k},\alpha)$ is not abelian.
\end{theorem}
\begin{proof}
First of all, notice that if $\alpha=\zeta_3$, then $G_\infty(f_{6k},\alpha)=G_\infty(f_{6k},0)$. Hence, it is enough to deal with the case where $\alpha-\zeta_3$ is a root of unity. Assume from now on that $G_\infty(f_{6k},\alpha)$ is abelian; we will derive a contradiction.

Let $\alpha-\zeta_3$ be a root of unity in $L$. Notice that the post-critical orbit of the pair $(f_{6k},\alpha)$ is $c_{1,\alpha}=\zeta_3-\alpha$, $c_{n,\alpha}=1+\zeta_3-\alpha$ for every $n\ge 2$. Now examining all the finitely many possibilities for $\alpha$, and using Lemma \ref{jones_lemma}, it is immediate to see that the extension $L_\infty(f_{6k},\alpha)/\Q$ can ramify only at primes dividing $6k$.

Now let $\zeta$ be a $6k$-th root of $\alpha-\zeta_3$ such that $\zeta-\zeta_3$ is not a root of unity (it is easy to verify that it is always possible to choose such $\zeta$). Since $G_\infty(f_{6k},\alpha)$ is abelian, it follows that $L(\sqrt[6k]{\zeta-\zeta_3})$ is also an abelian extension of $L$. But since $\zeta-\zeta_3$ is not a root of unity, we have $0<h(\zeta-\zeta_3)\le \log 2< 1$, and our above considerations on ramification allow us to apply Proposition \ref{AD for radicals} and conclude that $k\le 36$.

Now it is just a matter of excluding the finitely many remaining values of $k,\alpha$. This can be done via a Magma \cite{magma} computation exploiting Theorem \ref{1dimensionality}. Let us explain how we performed it. Since $\alpha-\zeta_3$ is a root of unity $\zeta$ in $L$, there is a $6k$-th root of $\zeta$ that is a $6k$-th, $12k$-th, $18k$-th or $36k$-th primitive root of unity, let us call it $\zeta'$. Now the group $G_\infty(f_{6k},\zeta')$ is a subquotient of $G_\infty(f_{6k},\alpha)$, and therefore if the latter is abelian then so is the former. Now Theorem \ref{1dimensionality} implies that the $\F_2$-vector space generated by $\zeta'-\zeta_3$ and $1+\zeta_3-\zeta'$ in $L(\zeta')^{\times}/{L(\zeta')^{\times}}^2$ must be $1$-dimensional. A Magma computation shows that the only case where this happens is when $k=1$. This corresponds to the case where $f=x^6+\zeta_3$ and $\alpha=\zeta_3+1$. However, in this case $-1$ is a root of $f-\alpha$, and hence the group $G_\infty(f_6,-1)$ must be abelian as well. One can prove that this does not happen, since $G_2(f_6,-1)$ is non-abelian.
\end{proof}

We remark that \cite[Theorem 1.2]{AZ} shows immediately that if $\alpha-\zeta_3$ is a root of unity and $G_\infty(f_{6k},\alpha)$ is abelian, then $k$ is bounded by a constant that is of the order of $3^{14}$. This leaves a large number of cases to be computer verified, and therefore we chose to use the stronger bound yielded by Proposition \ref{AD for radicals} in order to reduce the time needed for the verification.

\subsection{The remaining basepoints.}

Now we need to take care of all other basepoints, namely those such that $\alpha-\zeta_3$ is neither $0$ nor a root of unity. The idea is to combine Theorem \ref{1dimensionality} and Lemma \ref{abelianity_level1} in order to reduce to looking at polynomials of bounded degree. This is essentially the content of the following proposition.

\begin{proposition}\label{necessary_cases}
Let $k\ge 1$ be an integer and $f_{6k}=x^{6k}+\zeta_3$. Suppose $\alpha\in L$ is such that $\alpha-\zeta_3$ is neither $0$ nor a root of unity and that $f_{6k}-\alpha$ has no roots in $L$. Assume that $G_\infty(f_{6k},\alpha)$ is abelian. Then one of the following hold:
\begin{enumerate}[(i)]
\item There exists $\beta\in L^{\times}$ that is not a square and such that for every $i\in\{0,1,2\}$,  $x^2-\zeta_3^i\beta$ divides $f_{6k}-\alpha$;
\item There exists $\beta\in L^{\times}$ that is not a cube and such that for every $i\in \{0,1\}$, $x^3-(-1)^i\beta$ divides $f_{6k}-\alpha$;
\item There exists $\beta\in L^{\times}$ such that $\beta$ is not a square nor a cube, $\beta=1-t^6$ or $\beta=1/(1+t^6)$ for some $t\in L$ and $x^{6}-\beta$ divides $f_{6k}-\alpha$;
\item $x^6-8/9$ divides $f_{6k}-\alpha$.
\end{enumerate}
\end{proposition}

\begin{proof}
Write $6k=2^r3^sw$ with $r,s,w\geq 1$ and $2,3\nmid w$. Then by Theorem \ref{lang_theorem} it follows that $\alpha-\zeta_3\in L^w$ and hence $f-\alpha$ has a factor of the form $x^{2^r3^s}-\gamma$ for some $\gamma\in L^{\times}$. Now let $r'$ be the largest integer $\leq n$ such that $\gamma$ is a $2^{r'}$-th power and let $s'$ be the largest integer $\leq s$ such that $\gamma$ is a $3^{s'}$-th power. Then $g=x^{2^n3^m}-\beta$ is a factor of $f-\alpha$, where $n=r-r'$, $m=s-s'$ and $\beta=\gamma^{1/(2^{r'}3^{s'})}$. Notice that by our hypotheses $\beta$ is not zero or a root of unity. Then one of the following hold:
\begin{enumerate}[(A)]
\item $g=x^{2^n}-\beta$, where $n\geq 1$ and $\beta$ is not a square. In this case, $h=x^{2^n}-\zeta_3^i\beta$ divides $f-\alpha$, for every $i\in\{0,1,2\}$.
\item $g=x^{3^m}-\beta$, where $m\geq 1$ and $\beta$ is not a cube. In this case, $h=x^{3^m}+\beta$ divides $f-\alpha$ as well.
\item $g=x^{2^n3^m}-\beta$, where $m,n\geq 1$ and $\beta$ is not a square nor a cube;
\end{enumerate}

Now we show that if we are in case $(A)$, then $n=1$. Suppose that $n\ge 2$. By Lemma \ref{abelianity_level1}, we must have $\beta=\zeta_6t^2$, for some $t\in L^{\times}$. We want to apply Theorem \ref{1dimensionality} with $p=2$. In order to do so, we construct a sequence of polynomials of prime degree as follows: we let $g_1\coloneqq x^2-\beta$ and $g_i=x^2$ for every $i=2,\ldots,n$. Next, we let $6k=2q_1\ldots q_r$ where the $q_i$'s are primes, and we let $g_{n+1}=x^2+\zeta_3$ and $g_{n+1+i}=x^{q_i}$ for every $i=1,\ldots,r$. Now since $G_\infty(f,\alpha)$ is abelian, then so is $G_\infty(\{g_k\}_{k\ge 1},0)$, because the latter is a subquotient of the former. Now notice that the set $I\coloneqq \{i\in \N\colon \deg g_i=2\}$ certainly contains $1$ and $n+1+6k\N$. Theorem \ref{1dimensionality} then tells us that the set $\{-g^{(1)}(0),g^{(n+1)}(0),g^{(n+1+6k)}(0)\}=\{\beta,\zeta_3^2-\beta,\zeta_3-\beta\}$ is $1$-dimensional in $L^{\times}/{L^{\times}}^2$. This implies that there exist $y,z\in L$ such that one of the following hold:

$$(1):\begin{cases}\zeta_3-\zeta_6t^2=y^2 & \\
                        \zeta_3^2-\zeta_6t^2=z^2 & \\ \end{cases} \quad (2):\begin{cases}\zeta_3-\zeta_6t^2=y^2 & \\
                        \zeta_6(\zeta_3^2-\zeta_6t^2)=z^2 & \\ \end{cases}$$
                        
$$(3):\begin{cases}\zeta_6(\zeta_3-\zeta_6t^2)=y^2 & \\
                        \zeta_3^2-\zeta_6t^2=z^2 & \\ \end{cases} \quad (4):\begin{cases}\zeta_6(\zeta_3-\zeta_6t^2)=y^2 & \\
                        \zeta_6(\zeta_3^2-\zeta_6t^2)=z^2 & \\ \end{cases}$$
These systems all define genus 1 curves $E_1,\ldots,E_4$, as they are intersections of quadrics; one way to write a planar model for them is to write down an explicit parametrization for the four different conics appearing, and then equating them. So let
$$C_1:\zeta_3-\zeta_6T^2=Y^2, \qquad C_2:\zeta_3^2-\zeta_6T^2=Y^2,$$
$$C_3:\zeta_6(\zeta_3-\zeta_6T^2)=Y^2,\qquad C_4:\zeta_6(\zeta_3^2-\zeta_6T^2)=Y^2.$$
Parametrizations for the $T$-coordinates are, in order from $1$ to $4$:
$$\frac{-2\zeta_6X}{X^2+\zeta_6},\quad \frac{-2\zeta_3 X}{X^2+\zeta_6},\quad \frac{\zeta_6X^2+2X-\zeta_6}{X^2+1},\quad \frac{\zeta_6X^2-2\zeta_3X-\zeta_6}{X^2+1},$$
and the corresponding genus 1 curves are:
\begin{enumerate}
\item $X(Y^2+\zeta_6)=\zeta_6Y(X^2+\zeta_6)$;
\item $-2\zeta_6X(Y^2+1)=(\zeta_6Y^2-2\zeta_3Y-\zeta_6)(X^2+\zeta_6)$;
\item $(\zeta_6X^2+2X-\zeta_6)(Y^2+\zeta_6)=-2\zeta_3Y(X^2+1)$;
\item $(\zeta_6X^2+2X-\zeta_6)(Y^2+1)=(\zeta_6Y^2-2\zeta_3Y-\zeta_6)(X^2+1)$;
\end{enumerate}
Their rational points are studied in Appendix \ref{genus1}, and they altogether lead to $t=0$ or a root of unity, that is impossible because it would imply in turn that $\beta$ is $0$ or a root of unity.

Therefore $n=1$, and consequently $\beta$ cannot be a square since $f_{6k}-\alpha$ has no roots by hypothesis.

Next, we show that if we are in case $(B)$, then $m=1$. By Lemma \ref{abelianity_level1}, we have that $\beta=\zeta_6t^3$ or $\beta=\zeta_6^2t^3$ for some $t\in L^{\times}$. Now an analogous argument to the one used for case $(B)$ yields, via Theorem \ref{1dimensionality}, that the space spanned by $\{\beta,1-\beta,-1-\beta\}$ is $1$-dimensional modulo cubes. In particular, also the space spanned by $\beta$ and $1-\beta$ is $1$-dimensional modulo cubes. This implies the existence of some $y\in L$ such that one of the following holds:
\begin{enumerate}
\item $1-\zeta_6t^3=y^3$;
\item $\zeta_6(1-\zeta_6t^3)=y^3$;
\item $\zeta_6^2(1-\zeta_6t^3)=y^3$;
\item $1-\zeta_6^2t^3=y^3$;
\item $\zeta_6(1-\zeta_6^2t^3)=y^3$;
\item $\zeta_6^2(1-\zeta_6^2t^3)=y^3$.
\end{enumerate}

The above equations define elliptic curves that all have rank zero. Their rational points are listed in Appendix \ref{genus1}, and they all lead to $t$ being $0$ or a root of unity, that is impossible because it would imply that $\beta$ is $0$ or a root of unity.

Hence $m=1$, and therefore $\beta$ cannot be a cube since $f_{6k}-\alpha$ has no roots by hypothesis.

Now assume that we are in case $(C)$, so that $\beta$ is not a square nor a cube. First we prove that $n=m=1$ necessarily.

Assume by contradiction that $n>1$. By Lemma \ref{abelianity_level1}, we must have $\beta=\zeta\cdot \gamma^{2^{n-1}}$ for some $\zeta,\gamma\in L$ with $\zeta$ a root of unity. Since $\beta\notin L^2$, we can assume that $\zeta=\zeta_6$, without loss of generality. Since $n>1$, we can write $\beta=\zeta_6 t^2$ for some $t\in L^{\times}$. Since $\beta\notin L^3$, by the application of Theorem \ref{1dimensionality} already used above it follows that there exists $y\in L$ such that one of the following holds:
\begin{enumerate}
\item $1-\zeta_6 t^2=y^3$;
\item $\zeta_6 t^2(1-\zeta_6 t^2)=y^3$;
\item $\zeta_6^2t(1-\zeta_6t^2)=y^3$.
\end{enumerate}

The above equations define curves of genus $1$ (except for $(2)$ that has genus $2$ but is a double cover of a curve of genus $1$) whose rational points are listed in Appendix \ref{genus1}. All of the rational points have $t=0$ or a root of unity.

Now assume by contradiction that $m> 1$. By Lemma \ref{abelianity_level1} it must be $\beta=\zeta\cdot\gamma^{3^{m-1}}$ for some $\zeta,\gamma\in L^{\times}$ with $\zeta$ a root of unity. Then we can write $\beta=\zeta t^3$ for some $t\in L^{\times}$, and we can assume that $\zeta\in \{\zeta_6,\zeta_6^2\}$. Now since $\beta\notin L^3$ we can apply Theorem \ref{1dimensionality} again, and we see immediately that there exists $y\in L$ such that one of the following holds:
\begin{itemize}
\item $1-\zeta_6t^3=y^3$;
\item $\zeta_6(1-\zeta_6t^3)=y^3$;
\item $\zeta_6^2(1-\zeta_6t^3)=y^3$;
\item $1-\zeta_6^2t^3=y^3$;
\item $\zeta_6^2(1-\zeta_6^2t^3)=y^3$;
\item $\zeta_6(1-\zeta_6^2t^3)=y^3$.
\end{itemize}

As we already proved in case $(B)$, none of these relations can hold for $t$ not $0$ nor a root of unity.

Therefore if we are in case $(C)$, then $f_{6k}-\alpha$ has a factor of the form $x^6-\beta$, where $\beta\in L$ is not a square nor a cube in $L$. Applying Theorem \ref{1dimensionality} modulo squares and modulo cubes in the same way we did above, we see that $\{1-\beta,\beta\}$ is $1$-dimensional modulo squares and $1$-dimensional modulo cubes. This means that one of the following holds:

$$(a):\begin{cases}\beta(1-\beta)\in L^2 & \\
                        \beta(1-\beta)\in L^3 & \\ \end{cases} \quad (b):\begin{cases}\beta(1-\beta)\in L^2 & \\
                        1-\beta\in L^3 & \\ \end{cases} $$
                        
$$ (c):\begin{cases}1-\beta\in L^2 & \\
                        \beta(1-\beta)\in L^3 & \\ \end{cases}\quad (d):\begin{cases}1-\beta\in L^2 & \\
                        \beta^2(1-\beta)\in L^3 & \\ \end{cases}$$
                        
$$(e):\begin{cases}\beta(1-\beta)\in L^2 & \\
                        \beta^2(1-\beta)\in L^3 & \\ \end{cases} \quad (f):\begin{cases}1-\beta\in L^2 & \\
                        1-\beta\in L^3 & \\ \end{cases}$$
                        
Cases $(a)$ to $(d)$ imply, respectively, the existence of some $y\in L$ such that:
\begin{enumerate}
\setcounter{enumi}{3}
\item $t(1-t)=y^3$ with $t=\beta$;
\item $(1-t^3)t=y^2$, with $1-t^3=\beta$;
\item $t^2(1-t^2)=y^3$, with $1-t^2=\beta$;
\item $(1-t^2)t=y^3$, with $1-t^2=\beta$.
\end{enumerate}
Rational points of these curves are listed in Appendix \ref{genus1}. The only admissible case is $\beta=8/9$, coming from curve $(4)$.

If $(e)$ or $(f)$ hold, then $\beta=1/(1+t^6)$ or $\beta=1-t^6$, respectively, for some $t\in L$.
\end{proof}

The next elementary lemma reduces proving that $G_\infty(f_{6k},\alpha)$ is not abelian for some $k,\alpha$ to proving that a small extension of $L$ is not abelian.

\begin{lemma}\label{level_2_reduction}
Let $k\ge 1$ and $\alpha\in L$ be such that $G_\infty(f_{6k},\alpha)$ is abelian. Suppose that there exists $h(x)\in L[x]$ such that $h(x)$ divides $f_{6k}-\alpha$. Then $\gal(h(x^2+\zeta_3))$ is abelian.
\end{lemma}
\begin{proof}
If $G_\infty(f_{6k},\alpha)$ is abelian, then of course so is $\gal(f_{6k}^2-\alpha)$, since it is a quotient of the former. If $h(x)\in L[x]$ divides $f_{6k}-\alpha$, then also $\gal(h\circ f_{6k})$ is abelian. This means that if $\gamma$ is a root of $h$ then $L(\gamma)(\sqrt[6k]{\gamma-\zeta_3})/L$ is abelian. But then also $L(\gamma)(\sqrt{\gamma-\zeta_3})/L$ is abelian, and hence $\gal(h(x^2+\zeta_3))$ is abelian.
\end{proof}

Using Proposition \ref{necessary_cases} and Lemma \ref{level_2_reduction} we can now rule out all remaining basepoints $\alpha$.

\begin{lemma}\label{cases12}
Let $f_{6k}=x^{6k}+\zeta_3$. Suppose that $\alpha\in L$ is such that $\alpha-\zeta_3$ is not $0$ nor a root of unity, and that case $(i)$ or $(ii)$ of Proposition \ref{necessary_cases} holds. Then $G_\infty(f_{6k},\alpha)$ is not abelian.
\end{lemma}
\begin{proof}

\textbf{Case $(i)$}. In this case, there is a non-square $\beta\in L^{\times}$ such that $x^2-\zeta_3^i\beta\mid f_{6k}-\alpha$ for every $i\in\{1,2,3\}$. Let $g_i\coloneqq x^2-\zeta_3^i\beta$. By Lemma \ref{level_2_reduction}, it is enough to prove that if $\eta_i$ is a root of $(x^2+\zeta_3)^2-\zeta_3^i\beta$ for every $i\in \{1,2,3\}$, then the three extensions $L(\eta_i)/L$ cannot be all normal at the same time. It is a standard Galois theoretical fact that if $\eta_i=\sqrt{-\zeta_3+\sqrt{\zeta_3^i\beta}}$, then $L(\eta_i)/L$ is normal if and only if $\zeta_3^2-\zeta_3^i\beta\in L\left(\sqrt{\zeta_3^i\beta}\right)^2$. However, of course, $L\left(\sqrt{\zeta_3^i\beta}\right)=L(\sqrt{\beta})$ for every $i$. Hence for the three extensions $L(\eta_i)/L$ to be normal at the same time it is a necessary condition that
$$\prod_{i=1}^3(\zeta_3^2-\zeta_3^i\beta)\in L\left(\sqrt{\beta}\right)^2,$$
which is equivalent to asking that $1-\beta^3=(A+B\sqrt{\beta})^2$ for some $A,B\in L$. This quickly leads to $1-\beta^3=A^2$ or $1-\beta^3=B^2\beta$, that in turn lead to study rational points on the elliptic curves $y^2=x^3+1$ and $y^2=x^3-1$ (see Appendix \ref{genus1}). It turns out that the only admissible values for $\beta$ are $-2\zeta_3^i$ with $i\in \{1,2,3\}$; now one simply checks that for such values one of the extensions $L(\eta_i)/L$ is not normal.

\textbf{Case $(ii)$}. In this case there exists a non-cube $\beta\in L$ such that $x^3\pm\beta\mid f_{6k}-\alpha$. Again by Lemma \ref{level_2_reduction} it is enough to show that if $\eta$ is a root of $(x^2+\zeta_3)^3-\beta$ and $\eta'$ is a root of $(x^2+\zeta_3)^3+\beta$, the two extensions $L(\eta)/L$ and $L(\eta')/L$ cannot be both normal at the same time. Notice that $L(\sqrt[3]{\beta})=L(\sqrt[3]{-\beta})$ is a cubic extension of $L$. From now on, we fix a cubic root $\omega$ of $\beta$. By the standard argument we quoted in the proof of the previous case, $L(\eta)/L$ is normal if and only if $(-\zeta_3-\zeta_3^i\omega)(-\zeta_3-\zeta_3^j\omega)\in L(\omega)^2$ for every $i,j\in \{0,1,2\}$. Hence, in particular, if $L(\eta)/L$ is normal then $(1+\omega)(1+\zeta_3^2\omega)\in L(\omega)^2$. Analogously, if $L(\eta')/L$ is normal then $(1-\omega)(1-\zeta_3^2\omega)\in L(\omega)^2$. Multiplying the two relations, it follows that:
$$1+\zeta_3\beta\omega+\zeta_3^2\omega^2=(1-\omega^2)(1-\zeta_3\omega^2)=(a_0+a_1\omega+a_2\omega^2)^2$$
for some $a_0,a_1,a_2\in L$. This yields the system of equations:
$$\begin{cases}a_0^2+2\beta a_1a_2=1 & \\ \beta a_2^2+2a_0a_1=\zeta_3\beta & \\a_1^2+2a_0a_2=\zeta_3^2 & \\\end{cases}.$$

Linear combinations of equations 1-3 and 2-3 give the following system:

$$\begin{cases}a_1^2+2a_0a_2-\zeta_3^2a_0^2-2\zeta_3^2\beta a_1a_2=0 & \\\zeta_3\beta a_2^2+2\zeta_3a_0a_1-\beta a_1^2-2\beta a_0a_2=0 & \\\end{cases}.$$

One checks that $a_1=0$ leads to the only solution $a_0=a_1=a_2=0$, that in turn leads to the forbidden value $\beta=1$; hence we can replace $a_0$ with $a_0/a_1$ and $a_2$ with $a_2/a_1$. The first equation then implies that $a_2=\frac{\zeta_3^2a_0^2-1}{2(a_0-\zeta_3^2\beta)}$ (notice that it cannot be $a_0-\zeta_3^2\beta=0$, as this would lead $\beta$ to be a root of unity), and substituting into the second one we get:
$$C\colon 3a^4\beta - 4\zeta_3^2a^3\beta^2 -8\zeta_3^2a^3 + 18\zeta_3a^2\beta - 12a\beta^2 + 4\zeta_3^2\beta^3 -\zeta_3^2\beta=0.$$

This is a genus $1$ curve with finitely many rational points, see Appendix \ref{genus1}. The only admissible $\beta$'s turn out to be $\beta=\pm 1/2,\pm 2$, and to conclude the proof it is then enough to check that for such $\beta$'s one of the extensions $L(\eta)/L,L(\eta')/L$ is not normal.
\end{proof}

\begin{lemma}\label{cases34}
Let $f_{6k}=x^{6k}+\zeta_3$. Suppose that $\alpha\in L$ is such that $\alpha-\zeta_3$ is not $0$ nor a root of unity, and that case $(iii)$ or $(iv)$ of Proposition \ref{necessary_cases} holds. Then $G_\infty(f_{6k},\alpha)$ is not abelian.
\end{lemma}
\begin{proof}

\textbf{Case $(iii)$.} Let $\beta\in L^{\times}$ be such that $\beta$ is not a square, a cube nor a root of unity, $x^6-\beta$ divides $f-\alpha$ and $\beta=1-t^6$ or $\beta=1/(1+t^6)$ for some $t\in L$.

Let $M$ be the field generated by a root of $(x^2+\zeta_3)^6-\beta$. By Lemma \ref{level_2_reduction} it is enough to prove that $M$ is not normal over $L$. Let $\omega$ be a sixth root of $\beta$. Notice that $x^6-\beta$ is irreducible over $L$ by Lemma \ref{lang_lemma}, and hence $L(\omega)/L$ is a cyclic Galois extension of degree $6$. Since $M=L(\omega,\sqrt{-\zeta_3+\omega})$, we have by a standard argument in Galois theory that $M/L$ is normal if and only if for every $\sigma,\tau\in \gal(L(\omega)/L)$ we have $(-\zeta_3+\sigma(\omega))(-\zeta_3+\tau(\omega))\in L(\omega)^2$. In particular, since there exist $\sigma,\tau$ with $\sigma(\omega)=\zeta_3\omega$ and $\tau(\omega)=-\zeta_3\omega$, if $M/L$ is normal then there is $y\in L(\omega)$ such that $\zeta_3^2-\zeta_3^2\omega^2=y^2$. Now notice that $y$ can be written as $A+B\omega$, where $A=a_0+a_2\omega^2+a_4\omega^4$ and $B=b_1+b_3\omega^2+b_5\omega^4$ for some $a_0,\ldots,b_5\in L$, so that $A,B\in L(\omega^2)$. Hence
\begin{equation}\label{squares}
\zeta_3^2-\zeta_3^2\omega^2=A^2+B^2\omega^2+2AB\omega,
\end{equation}
and all terms of this equality live in $L(\omega^2)$, except possibly for $2AB\omega$. Since $\omega$ has degree $6$ over $L$ while $AB$ has at most degree $3$, the only possibility is that $AB=0$. Now we will analyze the two cases separately.

\textbf{Case $B=0$.} Equation \eqref{squares} implies that $A^2=\zeta_3^2-\zeta_3^2\omega^2$. Since $A^2=a_0^2+2a_2a_4\beta+(a_4^2\beta+2a_0a_2)\omega^2+(a_2^2+2a_0a_4)\omega^4$, we must have:

\begin{equation}\label{system1}
\begin{cases}a_0^2+2a_2a_4\beta=\zeta_3^2 & \\a_4^2\beta+2a_0a_2=-\zeta_3^2 & \\ a_2^2+2a_0a_4=0 & \\\end{cases}.
\end{equation}

Adding term to term the first and the second equation we get:
\begin{equation}\label{linear_combination}
a_0^2+2a_2a_4\beta+a_4^2\beta+2a_0a_2=0
\end{equation}

We can safely assume that $a_4\neq 0$, because otherwise the second and third equations together cannot be verified. Hence we can dehomogenize with respect to $a_4$ both \eqref{linear_combination} and the third equation of \eqref{system1}. The latter one says that $a_0=-a_2^2/2$, and substituting into \eqref{linear_combination} yields:
\begin{equation}\label{curve1}
a^4-4a^3+8a\beta+4\beta=0,
\end{equation}
where to ease the notation we just replaced $a_2$ with $a$.

Now recall that $\beta=1-t^6$ or $\beta=1/(1+t^6)$ for some $t\in L$.

Assume first that $\beta=1-t^6$. This yields:
\begin{equation}\label{squares_equation}
(a^2-2a-2)^2=t^6(8a+4).
\end{equation}
Notice that we can assume that $t\neq 0$, as otherwise $\beta=1$. It follows that $8a+4=d^2$ for some $d\in L$. Substituting in \eqref{squares_equation}, writing $t'$ for $4t$ and clearing denominators yields:
$$(d^4-24d^2-48)^2-d^2t'^6=0.$$
This means that either $t'$ or $-t'$ satisfies the equation:
\begin{equation}\label{genus_3}
d^4-24d^2-48=dt'^3,
\end{equation}
that defines a smooth curve of genus $3$. Now the isomorphism $x=6/d$ and $y=-3dt$ transforms curve \eqref{genus_3} into $y^3=x^4+18x^2-27$. Hence we can invoke Theorem \ref{chabauty}, and conclude that the set of $t\in L$ such that $\pm t$ satisfies equation \eqref{genus_3} for some $d\in L$ is: $\left\{\pm\zeta_3^i,\pm \frac{\zeta_3^i}{\sqrt{-3}}\right\}$. The corresponding values of $\beta$ are $0$ and $28/27$; the first one is forbidden by hypothesis and one checks that the extension of $L$ defined by $(x^2+\zeta_3)^6-28/27$ is not normal.

Next, let us go back to equation \eqref{curve1} and assume that $\beta=1/(1+t^6)$ for some $t\in L$. We then get:
$$t^6a^2(a^2-4a)+(a^2-2a-2)^2=0.$$
Since $ta$ cannot be $0$, it follows that $a^2-4a$ is minus a square in $L$, and hence there exists $d\in L$ such that $a=4/(1+d^2)$. Substituting and clearing denominators gives the relation $(2t)^6d^2=(d^4+6d^2-3)^2$, so that writing $t'$ for $2t$ we get that $t'$ or $-t'$ satisfies the relation:
\begin{equation}\label{genus_3b}
    t^3d=d^4+6d^2-3.
\end{equation}

The isomorphism $x=3/d$ and $y=-3dt$ transforms the above curve into $y^3=x^4-18x^2-27$. This is isomorphic, over $L(i)$, to $y^3=x^4+18x^2-27$, via mapping $(x,y)\mapsto (x,iy)$. Theorem \ref{chabauty} shows then that curve \eqref{genus_3b} has no $L$-rational (affine) points.

\textbf{Case $A=0$}. Equation \eqref{squares} implies that $B^2\omega^2=\zeta_3^2-\zeta_3^2\omega^2$. Since $B^2\omega^2=b_3^2\beta+2b_1b_5\beta+(b_1^2+2b_3b_5\beta)\omega^2+(b_5^2\beta+2b_1b_3)\omega^4$, we must have:

\begin{equation}\label{system2}
\begin{cases}b_3^2\beta+2b_1b_5\beta=\zeta_3^2 & \\b_1^2+2b_3b_5\beta=-\zeta_3^2 & \\ b_5^2\beta+2b_1b_3=0 & \\\end{cases}.
\end{equation}

Adding term by term the first two equations we get:
$$b_3^2\beta+2b_1b_5\beta+b_1^2+2b_3b_5\beta=0,$$
Here we can assume $b_3\neq 0$, since otherwise the system has no solutions, and set $b\coloneqq b_5/b_3$. Substituting the third equation of \eqref{system2} and canceling a factor $\beta$ since $\beta\neq 0$ yields:
\begin{equation}\label{curve2}
4+8b-4b^3\beta+b^4\beta=0.
\end{equation}
Now once again, $\beta=1-t^6$ or $\beta=1/(1+t^6)$ for some $t\in L$.

If $\beta=1-t^6$ we then get:
$$(b^2-2b-2)^2-t^6b^2(b^2-4b)=0;$$
since $tb\neq 0$ it follows that $b^2-4b=y^2$ for some $y\in L$. This implies that $b=4/(1-d^2)$ for some $d\in L$, and substituting in the above expression one gets:

$$(d^4-6d^2-3)^2-64d^2t^6=0,$$
so that after writing $t'$ for $2t$ we get that $\pm t'$ satisfies

\begin{equation}\label{genus_3c}
d^4-6d^2-3=dt'^3.
\end{equation}

The isomorphism $x=3/d$ and $y=-3dt$ transforms curve \eqref{genus_3c} into $y^3=x^4+18x^2-27$. Theorem \ref{chabauty} shows than that $t\in \left\{\zeta_3^i,\frac{\zeta_3^i}{\sqrt{-3}}\right\}$.
This yields once again $\beta=0$ or $\beta=28/27$, and as we said above, once checks that these values do not give rise to normal extensions of $L$.

Finally, if $\beta=1/(1+t^6)$ then equation \eqref{curve2} yields $(b^2-2b-2)^2=-(4+8b)t^6$, so that there exists $d\in L$ with $4+8b=-d^2$, and substituting gives, after clearing denominators,
$$(d^4+24d^2-48)^2=2^{12}d^2t^6.$$
Writing $t'$ for $4t$ we get that $\pm t'$ satisfies:
\begin{equation}\label{genus_3d}
    d^4+24d^2-48=dt^3.
\end{equation}
The isomorphism $x=6/d$ and $y=-3dt$ transforms curve \eqref{genus_3d} into $y^3=x^4-18x^2-27$. The same argument we used in case $B=0$ shows, via Theorem \ref{chabauty}, that curve \eqref{genus_3d} has no $L$-rational (affine) points.

\textbf{Case $(iv)$.} If $x^6-8/9$ divides $f-\alpha$, then one simply checks that $\gal((x^2+\zeta_3)^6-8/9)$ is non-abelian. The claim then follows again from Lemma \ref{level_2_reduction}.
\end{proof}

\begin{corollary}\label{final_cor}
Let $k\ge 1$ be an integer and $f_{6k}=x^{6k}+\zeta_3$. Suppose $\alpha\in L$ is such that $\alpha-\zeta_3$ is neither $0$ nor a root of unity and that $f-\alpha$ has no roots in $L$. Then $G_\infty(f_{6k},\alpha)$ is non-abelian.
\end{corollary}
\begin{proof}
Simply combine Proposition \ref{necessary_cases} and Lemmas \ref{cases12} and \ref{cases34}.
\end{proof}

We can now put all the pieces together and prove Theorem \ref{case_7}.

\begin{proof}[Proof of Theorem \ref{case_7}]
Let $\alpha\in L$ and $k\in \Z_{\ge 1}$. By Proposition \ref{exceptional}, the tree $T_\infty(f_{6k},\alpha)$ must contain a node $\gamma\in L$ such that $f_{6k}-\gamma$ has no roots in $L$.

Now the group $G_\infty(f_{6k},\gamma)$ is a subgroup of $G_\infty(f_{6k},\alpha)$. Assume that the latter is abelian. Then so is the former. If $\gamma-\zeta_3$ is a root of unity, by Theorem \ref{az_trick} the group $G_\infty(f_{6k},\gamma)$ is not abelian. Hence $\gamma-\zeta_3$ cannot be a root of unity. Since $f_{6k}-\gamma$ has no roots in $L$, we can apply Corollary \ref{final_cor} and conclude that $G_\infty(f_{6k},\gamma)$ is non-abelian, reaching a contradiction.
\end{proof}

\section{Finiteness of abelian basepoints}

In this section we will prove, using Theorem \ref{1dimensionality}, the following finiteness result.

\begin{theorem}\label{finiteness_thm}
Let $d\geq 2$ be an integer. Then there exists an explicitly computable finite set $U_d\subseteq \overline{\Q}$, depending only on $d$, such that if $K\subseteq \overline{\Q}$ is a number field, $u\in K\setminus U_d$ and $f=ux^d+1\in K[x]$, then there are only finitely many $\alpha\in K$ such that $G_\infty(f,\alpha)$ is abelian.
\end{theorem}
In fact $U_d$ will consists of the polynomials whose critical orbit contains at least $6$ distinct elements: this is a finite set for given $d$. As we shall see, for all the polynomials outside of $U_d$ we have also a reasonably explicit description for a finite set of $\alpha$'s outside of which we have non-abelian images. 

Notice that if $K$ is a number field and $f=ax^d+b\in K[x]$ is a unicritical polynomial with $ab\ne 0$, then $f$ is $K$-conjugate to $ab^{d-1}x^d+1$. On the other hand if $a\ne 0$ and $b=0$, so that $f=ax^d$, it is a straightforward consequence of \cite[Theorem 12]{andrews} that there are only finitely many $\alpha\in K$ such that $G_\infty(f,\alpha)$ is abelian.

In order to prove Theorem \ref{finiteness_thm}, we need a key lemma that shows that if a unicritical polynomial $f$ over a number field $K$ has a sufficiently long post-critical orbit, then there are only finitely many $\alpha\in K$ such that $G_\infty(f,\alpha)$ is abelian. First, we recall the following elementary lemma.

\begin{lemma}\label{superelliptic_genus}
Let $K$ be a number field, $d\ge 2$ be an integer and let $f(x),g(x)\in K[x]$ be coprime polyomials with $g$ separable and $\deg g\ge 5$. Then the curve $y^d=f(x)g(x)$ has finitely many $K$-rational points.
\begin{proof}
It follows either from Hurwitz formula or alternatively with an explicit construction of unramified coverings, that the genus of this curve is at least $2$. We now conclude with Faltings' theorem.    
\end{proof}
\end{lemma}
In order to prove Theorem \ref{finiteness_thm} it suffices to prove the following proposition.
\begin{proposition} \label{finiteness proposition} Let $K$ be a number field and $d\ge 2$ be an integer. Let $f=ux^d+1\in K[x]$ be a PCF polynomial with $u\ne 0$ such that the elements $f(0),f^2(0),\ldots,f^6(0)$ are all distinct. Then there are only finitely many $\alpha\in K$ such that $G_\infty(f,\alpha)$ is abelian.

\end{proposition}

Let us first articulate how Proposition \ref{finiteness proposition} yields Theorem \ref{finiteness_thm},
\begin{proof}[Proof of Theorem \ref{finiteness_thm}]
The adjusted post-critical orbit $\{c_{i,0}(f)\}_{i\ge 1}$ has less than $6$ elements if and only if $u$ satisfies a relation of the form $f^{i}(0)=f^{j}(0)$ for some $i,j\le 6$ and $i\ne j$. Since for a given $d$ there are only finitely many such relations, there are also finitely many such $u$'s. For $u$'s outside of this finite set, Proposition \ref{finiteness proposition} shows that $G_\infty(ux^d+1,\alpha)$ is abelian at most for finitely many $\alpha$.
\end{proof}


The rest of this section is therefore devoted to proving Proposition \ref{finiteness proposition}.\\
\begin{proof}[Proof of Proposition \ref{finiteness proposition}]
Observe that there is no loss of generality in assuming that $\zeta_d$ is in $K$: if we can prove the conclusion for $K(\zeta_d)$ it clearly follows for $K$, as any abelian representation for $K$ restricts to an abelian representation for $K(\zeta_d)$. 

Next, let $\text{PrePer}_f(K)$ be the set of $K$-rational preperiodic points for $f$; this is a finite set thanks to Northcott theorem. From now on, we will only consider baspoints $\alpha$ such that
$$\alpha\in K\setminus \text{PrePer}_f(K).$$
Then the following hold true:

\begin{enumerate}[(A)]
    \item no two nodes of the tree $T_\infty(f,\alpha)$ are equal;
    \item the set of nodes of $T_\infty(f,\alpha)$ that belong to $K$ is finite.
\end{enumerate}

In fact, $(A)$ holds true because two nodes being equal at different levels imply that $\alpha$ is preperiodic for $f$, and two nodes being equal at the same level imply via Lemma \ref{jones_lemma} that $\alpha$ is in the post-critical orbit of $f$. But this is forbidden because $f$ is PCF, and hence if $\alpha$ is in the post-critical orbit then it is also pre-periodic. On the other hand, $(B)$ holds true by the Northcott property and the fact that $(A)$ holds true.

Now let $\text{Max}_f(\alpha)$ be the set of nodes of the tree $T_{\infty}(f,\alpha) \cap K$ such that all their descendants are not in $K$. Moreover, let $\ell_{\max}(\alpha)$ be the highest level of $T_\infty(f,\alpha)$ containing a point of $\text{Max}_f(\alpha)$. Thanks to $(B)$, the level $\ell_{\max}(\alpha)$ is well-defined.

The logic of the proof consists of the following $3$ steps:

$(1)$ Whenever $w$ in $\text{Max}_f(\alpha)$ does not come from a point $S_f$ occurring as $x$-coordinate of one of finitely many curves, depending only on $f$, we can for sure guarantee that $G_{\infty}(f,\alpha)$ is not abelian.

$(2)$ We are left with the case that all of $\text{Max}_f(\alpha)$ is in a certain finite set $S_f$ constructed in $(1)$. This forces $\alpha$ to be in finitely many possible $f$-\emph{orbits}, not yet what we are after: that is, we need an upper bound on how far $\alpha$ can be in these orbits. In other words, we need an upper bound (depending only on $f$) on $\ell_{\max}(\alpha)$.

$(3)$ We show a lower bound on $\#\text{Max}_f(\alpha)$ purely in terms of a strictly increasing function of $\ell_{\text{max}}(\alpha)$. This says that if $\ell_{\text{max}}(\alpha)$ is sufficiently big and all of $\text{Max}_f(\alpha)$ is inside $S_f$, then there must be a repetition in $\text{Max}_f(\alpha)$. But this contradicts $(A)$. Therefore $\ell_{\max}(\alpha)$ must be bounded in terms of $f$.

Let us first give the basic claim for step $(3)$. 

\emph{Claim:} We have that $\#\text{Max}_f(\alpha) \geq \ell_{\text{max}}+1$. \\
\emph{Proof of Claim:} The key property of $\text{Max}_f(\alpha)$ is that if a point is in $K$, then all of its siblings are, because the polynomials $x^d-\beta$ have one root if and only if they split completely, since $\zeta_d$ is in $K$. From here we argue as follows. Start with a node $w$ in $\text{Max}_f(\alpha)$ at level $\ell_{\max}(\alpha)$. Let us write the word
$$w\coloneqq x_{g(\ell_{\text{max}})} \ldots x_{g(1)}. 
$$
Now pick $h\coloneqq h(\ell_{\text{max}})$ in $\{1,\ldots,d\}$ different from $g(\ell_{\text{max]}})$ (recall that $d \geq 2$). Then, by the above observation we see that the tree below
$$x_{h(\ell_{\text{max}})}x_{g(\ell_{\text{max}}-1)} \ldots x_{g(1)},
$$
must contain a point of $\text{Max}_f(\alpha)$. Now continue like this with 
$$x_{h(\ell_{\text{max}})} \ldots x_{h(i+1)}x_{g(i)} \ldots x_{g(1)},
$$
down to 
$$x_{h(\ell_{\text{max}})} \ldots x_{h(1)}. 
$$
The result is a collection of $\ell_{\max}$ \emph{disjoint trees}, each of them necessarily possessing an element of $\text{Max}_f(\alpha)$. Now observing that any sibling of $w$ is also in $\text{Max}_f(\alpha)$, we conclude that there are at least
$$\ell_{\text{max}}+1
$$
elements in $\text{Max}_f(\alpha)$. This is precisely the desired conclusion and ends the proof of the claim.

We now implement step $(1)$. Since $\alpha\notin \text{PrePer}_f(K)$, no element of the tree lies in the post-critical orbit of $f$. Hence for every $\beta \in\text{Max}_f(\alpha)$ we have an arboreal representation in $\Omega_{\infty}(\{d\})$, and one where $0\ne {\phi_1}_{|G_{\infty}(f,\beta)} \in \Z/d\Z$. Let then $d_1\coloneqq {\phi_1}_{|G_{\infty}(f,\beta)}$: by Theorem \ref{1dimensionality_general} the abelianity of $G_\infty(f,\beta)$ implies that there is an element $\gamma_1(\alpha)\in K^{\times}$ and exponents $e_2,\ldots,e_6\in \{0,\ldots,d-1\}$ such that
$$\widetilde{c}_{i,\alpha}(f)=\gamma_1(\alpha)^{e_i} \mbox{ for }i=2,\ldots,6$$
and hence such that
$$\gamma_1(\alpha)^{de_i-e_i}\widetilde{c}_{i,\alpha}(f)=\gamma_1(\alpha)^{e_id} \mbox{ for }i=2,\ldots,6$$
Unraveling the definitions of $\gamma_1(\alpha)$ and $\widetilde{c}_{i,\alpha}$, this shows that $\alpha$ is the $x$-coordinate of a point on one of finitely many curves of the form:
$$y^d=c(f(0)-x)^e\prod_{i=2}^6(f^i(0)-x),$$
where $c\in K$ and $e\in \{0,\ldots,d\}$. By Lemma \ref{superelliptic_genus}, these curves all have finitely many $K$-rational points; as a consequence, we have produced a finite set $S_f$, depending only on $f$, such that either
$$\text{Max}_f(\alpha) \subseteq S_f,
$$
or $G_{\infty}(f,\alpha)$ is not abelian. But now, thanks to the Claim above, we see that if $\ell_{\text{max}} \geq \#S_f$, then there must be a repetition among the nodes of $\text{Max}_f(\alpha)$, contradicting $(A)$. 

All in all, we have proved that outside of the finite set 
$$\{f^n(s)\colon s \in S_f, n \leq \#S_f\} \cup \text{PrePer}_f(K),
$$
we have all $\alpha$'s for which $G_{\infty}(f,\alpha)$ is non-abelian, which gives the desired conclusion. 
\end{proof}
We remark that if $d$ is odd one can reduce the number of distinct points in the post-critical orbit of $f$ required to make Proposition \ref{finiteness proposition} work, through a refinement of Lemma \ref{superelliptic_genus}. When $d=2$, one has the following immediate corollary.
\begin{corollary}
Let $K$ be a number field and $u\in K$ such that
$$[\Q(u):\Q]\notin \{1,2,3,6,7,8,12,15\}.$$
Then there are at most finitely many $\alpha\in K$ such that $G_\infty(ux^2+1,\alpha)$ is abelian.
\end{corollary}

\appendix
\section{\texorpdfstring{$S$}{}-integral points on genus \texorpdfstring{$0$}{} curves}

Let $K$ be a number field, $C\colon F(x,y)=0$ a geometrically irreducible genus $0$ curve (possibly singular), where $F(x,y)\in K[x,y]$ has degree $N\geq 2$, and $S$ be a non-empty finite set of places of $K$. We will denote by $\O_{K,S}$ the ring of $S$-integers of $K$. When the function field of $C$ has at least three infinite places, the set of affine $S$-integral points on $C$ is finite, due to a theorem of Siegel (see for example \cite[Theorem~5.1]{lang2}). In this section we will briefly describe an algorithm to compute this set in the aforementioned setting, following \cite{alvanos}. This requires the knowledge of a smooth $S$-integral point. Next, we will apply it to the curves \eqref{genus0_1_curve} and \eqref{genus0_2_curve} to determine their set of $2$-integral points.

From now on, we will let $P_1,P_2,P_3$ be three distinct infinite places of $K(C)$. We will assume that they are defined over $K$, since our working examples both enjoy this property.

\textbf{Step 1.} Compute a basis $f_i$ of the 1-dimensional Riemann-Roch space $L(P_3-P_i)$, for $i=1,2$. We can choose the $f_i$'s to have the form $a_i(x,y)/b_i(x)$ where $\deg_ya_i(x,y)<N$. By the Riemann-Roch theorem, the space $L(P_3)$ is 2-dimensional, and therefore there one can compute $c_1,c_2\in K$ such that
\begin{equation}\label{unit_equation}
c_1f_1+c_2f_2=1.
\end{equation}

\textbf{Step 2.} Since $\deg f_i=1$, the function field $K(C)$ coincides with $K(f_i)$ for any $i=1,2$. It follows that $[K(x,f_i):K(x)]=N$; moreover since the zeroes and the poles of the $f_i$'s are at infinity, $f_i$ is a root of an irreducible polynomial of the form $m_i(t)=\sum_{n=0}^{N} A_{n,i}(x)t^{n}\in K(x)[t]$, where $A_{n,i}(x)\in K[x]$ for every $n$ and  $A_{0,i},A_{N,i}\in K$. Now compute $\alpha_i,\beta_i\in \O_{K,S}$ such that $\alpha_if_i$ and $\beta_i/f_i$ are both integral over $\O_{K,S}[x]$.

\textbf{Step 3.} Let $(v,w)$ be a smooth $S$-integral point of $C$. Since $\O_{K,S}$ is integrally closed in $K$, it follows that $\alpha_if_i(v,w),\beta_i/f_i(v,w)\in\O_{K,S}$ and that $\alpha_if_i(v,w)$ divides $\alpha_i\beta_i$. Compute a maximal set $A_i$ of pairwise non-associate elements of $\O_{K,S}$ dividing $\alpha_i\beta_i$. Then $\alpha_if_i(v,w)=k_iu_i$ for some $k_i\in A_i$ and $u_i\in \O_{K,S}^{\times}$.

\textbf{Step 4.} Determine, for every $(k,k')\in A_1\times A_2$, the finite set $S(k,k')\subseteq {\O_{K,S}^{\times}}^2$ of solutions to the $S$-unit equation
$$\frac{c_1k}{\alpha_1}\cdot X+\frac{c_2k'}{\alpha_2}\cdot Y=1$$
By \eqref{unit_equation}, we have that $(u_1,u_2)\in S(k_1,k_2)$.

\textbf{Step 5.} For every $(k,k')\in A_1\times A_2$ and every $(u,u')\in S(k,k')$, compute the resultant $R_{k,u}(x)$ of $F(x,y)$ and $\alpha_1a_1(x,y)-kub_1(x)$ with respect to $y$. Since $\deg_ya_1(x,y)<N$, the polynomial $R_{k,u}(x)$ is non-zero; notice that $R_{k_1,u_1}(v)=0$. Therefore the set of all $v$ that are $x$-coordinate of an $S$-integral smooth point of $C$ is contained in the set of all $v\in \O_{K,S}$ that are roots of some $R_{k,u}$.

\begin{theorem}\label{genus0_1}
Let $K\coloneqq \Q(i)$. The set of $\O_{K,2}$-rational affine points on the curve:
$$C_1\colon (x^2-y^2)^2 - (2i + 4)x^2 + (1-2i)y^2=0$$
is given by:
$$\left\{(0,0),\left(\pm \frac{3+i}{8}, \pm \frac{7+9i}{8}\right)\right\}$$
\end{theorem}
\begin{proof}
We need to apply the algorithm described above with $S=\{(1+i)\}$. All steps have been implemented in Magma \cite{magma}, except for the computation of the solutions of the $S$-unit equations that was carried out in Sage \cite{sagemath} thanks to the implementation described in \cite{alvarado}.

One readily verifies that the only singular affine point of $C_1$ is $(0,0)$ and that $C$ has genus $0$. There are two singular points at infinity $Q_1=(1:1:0)$ and $Q_2=(1:-1:0)$; above each $Q_i$ there are, in the function field $K(C_1)$, two places of degree 1. Let $P_1,P_2$ be the ones lying above $Q_1$ and $P_3,P_4$ be the ones lying above $Q_2$. Then
$$f_1=\frac{\frac{2-i}{10}y^3 + \frac{2-i}{10}xy^2 + \left(\frac{i - 2}{10}x^2 + 1\frac{1}{2}x - \frac{i}{2}\right)y + \frac{i - 2}{10}x^3 + \frac{1}{2}x^2 + x}{x}$$
and
$$f_2=\frac{\frac{1+2i}{5}y^3 + \frac{1+2i}{5}xy^2 + \left(-\frac{1+2i}{5}x^2 + ix + 1\right)y -\frac{1+2i}{5}x^3 + ix^2 + x}{x}$$
are bases for the 1-dimensional Riemann-Roch spaces $L(P_3-P_1)$ and $L(P_3-P_2)$, respectively. Moreover, one sees that $c_1f_1+c_2f_2=1$ for $c_1=(4-2i)/5$ and $c_2=(1+2i)/5$.

The minimal polynomials for $f_1,f_2$ over $K[x]$ are
$$m_1=t^4 - (2x + (i + 4))t^3 + ((i + 4)x + 3i + 6)t^2 - \left((i + 2)x+ \frac{7+6i}{2}\right)t + \frac{3+4i }{4}$$
and
$$m_2=t^4 - (4ix + 2)t^3 + (8i + 8)xt^2 - ((4i + 8)x + 4i - 2)t + 4i + 3,$$
respectively. It is immediate therefore to see that $\alpha_1=\alpha_2=1$, $\beta_1=\beta_2=i+2$ and consequently $A_1=A_2=\{1,i+2\}$. The following are all solutions to the $S$-unit equations $c_1kX+c_2k'Y=1$, for $(k,k')\in A_1\times A_2$.

\begin{center}
		\begin{tabular}{|c || c|}
		    \hline
		    $(X,Y)$ & $(k,k')$\\
		    \hline
		    \hline
				$(1/2(i + 1), -i) $ & $(1,1)$\\
				\hline
				$(i/2, -2i)$ &  $(1,1)$\\
				\hline
				$(i + 1, -1)$ &  $(1,1)$\\
				\hline
				$(1/2, -i + 1)$ &  $(1,1)$\\
				\hline
				$ (1, 1)$ & $(1,1)$\\
				\hline 
				$( -i/2, -2i + 2)$ & $(1,1)$\\
				\hline
				$(1, i )$ &  $(i + 2, i + 2)$\\
				\hline
				$( 1/4(-i + 1), 1/2(-i + 1))$ &  $(i + 2, i + 2)$\\
				\hline
				$( i/2, -i - 1)$ &   $(i + 2, i + 2)$\\
				\hline
				$(1/2(i + 1), -1)$ & $(i + 2, i + 2)$\\
				\hline
				$(1/4, -i/2)$ & $(i + 2, i + 2)$\\
				\hline
				$( -1/2, -2i)$ & $(i + 2, i + 2)$\\
				\hline
				$(1/4(i + 1), 1/2(-i - 1))$ & $(i + 2, i + 2)$\\
				\hline
				$( 1/2(-i + 1), 1)$ & $(i + 2, i + 2)$\\
				\hline
				$(-i/2, -i + 1)$ & $(i + 2, i + 2)$\\
				\hline
		  \end{tabular} 
		  \end{center}
		  
Finally, we perform Step 5 of the algorithm and we get the full list of $S$-integral points.
\end{proof}

Analogous computations lead to a proof of the following theorem.

\begin{theorem}\label{genus0_2}
Let $K\coloneqq \Q(i)$. The set of $\O_{K,2}$-rational affine points on the curve:
$$C_2\colon 4y^4 -4x^2y^2 + (4i + 3)x^2 + (-4i - 8)y^2 + 4i + 3=0$$
is given by:
$$\left\{(\pm i,0),\left(\pm \frac{7i - 5}{8}, \pm \frac{2-i}{2}\right)\right\}.$$
\end{theorem}
\section{Genus \texorpdfstring{$1$}{} curves with finitely many rational points}\label{genus1}
In this appendix we list all rational points of curves that appear throughout the proofs of Section \ref{cyclotomic}. In order to ease the reader's effort, we divide them accordingly to the part of the proof of Proposition \ref{necessary_cases} or Lemma \ref{cases12} where they appear. All computations have been performed via Magma \cite{magma}.
\subsection{Proposition \ref{necessary_cases}, case \texorpdfstring{$(A)$}{}} Variable ordering: $X,Y$.

\begin{enumerate}
\item $X(Y^2+\zeta_6)=\zeta_6Y(X^2+\zeta_6)$. This is a rank $0$ elliptic curve whose affine rational points are:
$$\{(0,0),(\pm \zeta_3,\pm\zeta_3)\}.$$
\item $-2\zeta_6X(Y^2+1)=(\zeta_6Y^2-2\zeta_3Y-\zeta_6)(X^2+\zeta_6)$. This is a singular curve of genus $1$ (whose singular points are infinite) without any affine rational point. This can be seen by finding an elliptic model of rank $0$ over $\Q(\zeta_{12})$.
\item $(\zeta_6X^2+2X-\zeta_6)(Y^2+\zeta_6)=-2\zeta_3Y(X^2+1)$. Same as above.
\item $(\zeta_6X^2+2X-\zeta_6)(Y^2+1)=(\zeta_6Y^2-2\zeta_3Y-\zeta_6)(X^2+1)$. This is a rank $0$ elliptic curve whose affine rational points are:
$$\{(\zeta_3,0),(0,0),(\zeta_6,\zeta_3),(0,\zeta_6),(\zeta_3,\zeta_6)\}.$$
\end{enumerate}

\subsection{Proposition \ref{necessary_cases}, case \texorpdfstring{$(B)$}{}} Variable ordering: $t,y$.

\begin{enumerate}
\item $1-\zeta_6t^3=y^3$. This is a rank $0$ elliptic curve whose affine rational points are:
$$\{(0,1),(0,\zeta_3),(0,\zeta_3^2)\}.$$
\item $\zeta_6(1-\zeta_6t^3)=y^3$. This is a rank $0$ elliptic curve whose affine rational points are:
$$\{(1,-\zeta_6),(-\zeta_6,1),(\zeta_3,\zeta_3),(\zeta_3,-\zeta_6),(1,1),(-\zeta_6,\zeta_3),(1,\zeta_3),(-\zeta_6,-\zeta_6),(\zeta_3,1)\}.$$
\item $\zeta_6^2(1-\zeta_6t^3)=y^3$. This is a rank $0$ elliptic curve without affine rational points.
\item $1-\zeta_6^2t^3=y^3$. This is a rank $0$ elliptic curve whose affine rational points are:
$$\{(0,\zeta_3),(0,1),(0,-\zeta_6)\}.$$

\item $\zeta_6(1-\zeta_6^2t^3)=y^3$. This is a rank $0$ elliptic curve without rational affine points.

\item $\zeta_6^2(1-\zeta_6^2t^3)=y^3$. This is a rank $0$ elliptic curve whose affine rational points are:
$$\{(-1,-1),(-\zeta_3,\zeta_6),(\zeta_6,-\zeta_3),(-1,-\zeta_3),(-\zeta_3,-1),(\zeta_6,\zeta_6),(\zeta_6, -1),(-1,\zeta_6),(-\zeta_3,-\zeta_3)\}.$$
\end{enumerate}

\subsection{Proposition \ref{necessary_cases}, case \texorpdfstring{$(C)$}{}} Variable ordering: $t,y$.

\begin{enumerate}
\item $1-\zeta_6t^2=y^3$. This is a rank $0$ elliptic curve whose affine rational points are:
$$\{(0,1),(0,\zeta_3),(0,\zeta_3^2)\}.$$

\item $\zeta_6t^2(1-\zeta_6t^2)=y^3$. This is a genus $2$ quadratic cover of the rank $0$ elliptic curve $\zeta_6t(1-\zeta_6t)=y^3$. Its affine rational points are:
$$\{(0,0),(\pm1,\zeta_3),(\pm\zeta_3,1),(\pm1,-\zeta_6),(\pm\zeta_3,-\zeta_6),(\pm1,1),(\pm\zeta_3,\zeta_3)\}.$$

\item $\zeta_6^2t(1-\zeta_6t^2)=y^3$. This rank $0$ elliptic curve whose only rational affine point is $(0,0)$.

\item $t(1-t)=y^3$. This is a rank $0$ elliptic curve whose affine rational points are:
$$\{(0,0),(-\zeta_3,\zeta_3),(\zeta_6,-\zeta_6),(-\zeta_3,1),(\zeta_6,1),(1,0),(-\zeta_3,-\zeta_6),(\zeta_6,\zeta_3)\}.$$

\item $(1-t^3)t=y^2$. This is an elliptic curve whose affine rational points are:
$$\{(1,0),(0,0),(-\zeta_6,0),(\zeta_3,0)\}.$$

\item $t^2(1-t^2)=y^3$. This is a quadratic cover of the elliptic curve $t(1-t)=y^3$ studied above, and its affine rational points are:
$$\{(0,0),(\pm1,0)\}.$$

\item $(1-t^2)t=y^3$. This is an elliptic curve whose affine rational points are:
$$\{(0,0),(\pm 1,0),(1/3 ,2/3\zeta_3^i),(-1/3,-2/3\zeta_3^i)\colon i=0,1,2\}.$$
\end{enumerate}

\subsection{Lemma \ref{cases12}} Variable ordering: $x,y$.

\begin{itemize}
\item $y^2=x^3+1$. This is a rank $0$ elliptic curve whose affine rational points are:
$$\{(0,\pm1),(-\zeta_3^i,0),(2\zeta_3^i,\pm3)\colon i=1,2,3\}.$$

\item $y^2=x^3-1$. This is an elliptic curve whose affine rational points are:
$$\{(\zeta_3^i,0)\colon i\in\{1,2,3\}\}.$$

\item $3y^4x - 4\zeta_3^2y^3x^2 - 8\zeta_3^2y^3 + 18\zeta_3y^2x - 12yx^2 +4\zeta_3^2x^3 - \zeta_3^2x=0.$ This is a singular genus $1$ curve that is birational to a rank $0$ elliptic curve. Its affine rational points are:
$$\{(-1,\zeta_6),(1,-\zeta_6),(\pm 1/2,0),(0,0),(-2,\zeta_6),(2,-\zeta_6)\}.$$ 
\end{itemize}

\section{Rational points on \texorpdfstring{$y^3=x^4+18x^2-27$}{}}

In this appendix we prove the following theorem, using Siksek's work \cite{siksek} and the Magma implementation of Balakrishnan-Tuitman's algorithm \cite{balatu} that is available at \cite{balatu2}.

\begin{theorem}\label{chabauty}
The set of $\Q(\zeta_3,i)$-rational points on the curve $y^3=x^4+18x^2-27$ is given by:
$$S\coloneqq\{\infty,(0,-3\zeta_3^i),(\pm 1,-2\zeta_3^i),(\pm 3,6\zeta_3^i),(\pm(3+6\zeta_3),6\zeta_3^i)\colon i=1,2,3\}.$$
\end{theorem}

The proof of Theorem \ref{chabauty} involves an explicit form of Chabauty's criterion for curves over number fields in a single unit ball, proved in \cite{siksek}. Let us describe the algorithm. Suppose that $C$ is a projective curve of genus $g\ge 2$ over a number field $K$ of degree $d$. Let $v$ be a place of $K$, and suppose that $\mathscr{C}_v$ is a minimal regular proper model at $v$. Let $K_v$ be the completion of $K$ at $v$ and let $Q\in C(K_v)$. As the model is regular, the reduction $\widetilde{Q}$ of $Q$ on the special fiber is a smooth point. Let $s_Q\in K_v(C)$ be a rational function such that the maximal ideal of $\mathcal O_{\mathscr{C}_v,\widetilde{Q}}$ is generated by $s_Q$ and $\pi$, where $\pi$ is a uniformizing element for $K_v$. Then $t_Q\coloneqq s_Q-s_Q(Q)$ is referred to as a \emph{well-behaved uniformizer} at $Q$.
The \emph{$v$-adic unit ball around $Q$} is:
$$\mathscr{B}_v\coloneqq \{P\in C(K_v)\colon \widetilde{P}=\widetilde{ Q}\}.$$

\begin{lemma}[{{\cite[Lemma 3.2]{siksek}}}]\label{local_integrals}
Let $v$ be a place of $K$ such that the rational prime below it is odd and unramified in $K$. Let $Q\in C(K_v)$ and let $t_Q\in K_v(C)$ be a well-behaved uniformizer at $Q$. Let $\omega\in \Omega_{\mathscr{C}_v/\mathcal{O}_v}$. Then there is a power series $\phi(x)=\alpha_1(Q)x+\alpha_2(Q)x^2+\ldots \in xK_v\llbracket x\rrbracket$ that converges for every $z\in \pi\mathcal O_v$ and such that

$$\int_Q^P\omega=\phi(z)$$
for all $P\in \mathscr{B}_v(Q)$, where $z=t_Q(P)$. Moreover,the coefficient $\alpha_1(Q)$ is given by $(\omega/dt_Q)(Q)$ and it belongs to $\mathcal O_v$.
\end{lemma}

The integral in the lemma above is a Coleman integral. For the theory of these objects, see for example \cite{coleman1}.

Now let $r$ be the rank of the Jacobian variety of $J$. Assume that:
\begin{equation}\label{rank_genus}
r\le [K:\Q](g-1).
\end{equation}
Let $p$ be an odd rational prime, and suppose that $p$ is unramified in $K$ and $C$ has good reduction at every prime $v\mid p$. Let $Q\in C(K)$ be a rational point. Let
$$\mathscr{B}_p(Q)=\prod_{v\mid p}\mathscr{B}_v.$$ 

Let $\{D_1,\ldots,D_r\}$ be a basis of a finite index subgroup of $J(K)$ and for every $v\mid p$ let $\{\omega_{v,1},\ldots,\omega_{v,g}\}$ be an $\mathcal O_K$-basis of $\Omega_{\mathscr{C}_v/\mathcal O_v}$. Once we have chosen, once and for all, a $\Z_p$-basis $\mathcal B_v$ for $\O_v$, and letting $d_v\coloneqq [K_v:\Q_p]$, for every $i$ we can form a $d_v\times r$ matrix $T_{v,i}$ with entries in $\Q_p$ that expresses $\int_{D_1}\omega_{v,i},\ldots,\int_{D_r}\omega_{v,i}$ in such basis. Putting all of this information together we get a $d_vg\times r$ matrix
$$T_v\coloneqq\left(\begin{array}{c}T_{v,1}\\ \vdots\\ T_{v,g}\end{array}\right)$$
with entries in $\Q_p$, and putting all the $T_v$'s together we get a $dg\times r$ matrix
\begin{equation}\label{matrixT}
T\coloneqq\left(\begin{array}{c}T_{v_1}\\ \vdots\\ T_{v_n}\end{array}\right),
\end{equation}
where $v_1,\ldots,v_n$ are the places above $p$.

On the other hand if $Q\in C(K)$ is a known rational point and $v\mid p$, we can form a $d_vg\times d_v$ matrix $A_v$, that depends on $Q$, with entries in $\Z_p$ by writing each $\alpha=(\omega_{v,i}/dt_Q)\in \O_v$ with respect to the basis $\mathcal B_v$. The matrix $A_Q$ is then defined by

\begin{equation}\label{matrixA}
A_Q\coloneqq\left(\begin{array}{cccc} A_{v_1} & 0 & \ldots & 0\\ 0 & A_{v_2} & \ldots & 0\\ \vdots & \vdots & \ddots & \vdots\\ 0 & 0 & \ldots & A_{v_n}\end{array}\right)\in \text{Mat}_{dg\times d}(\Z_p).
\end{equation}

Finally, let $a\in \N$ such that $p^aT\in \text{Mat}_{dg\times r}(\Z_p)$ and let $U\in GL_{dg\times dg}(\Z_p)$ be such that $Up^aT$ is in Hermite normal form. Let $h$ be the number of zero rows of $Up^aT$ (notice that under assumption \eqref{rank_genus}, we have $h\ge d$). These are the last $h$ rows. Let $M_p(Q)$ be the matrix given by the last $h$ rows of $UA_Q$.

\begin{theorem}[{{\cite[Theorem 2]{siksek}}}]\label{chabauty_criterion}
With the notations above, let $\widetilde{M}_p(Q)$ be the reduction modulo $p$ of $M_p(Q)$. If $\widetilde{M}_p(Q)$ has rank $d$, then $C(K)\cap \mathscr{B}_p(Q)=\{Q\}$.
\end{theorem}

\begin{proof}[Proof of Theorem \ref{chabauty}]
Throughout the proof, we let $K\coloneqq \Q(\zeta_3,i)$.

A quick search for point of small height returns the list $S$; we need to show that these are the only rational points.

First, we need to find a basis for a finite index subgroup of $J(K)$. The Magma intrinsic \texttt{RankBounds} proves that $J(\Q)$ has rank $2$, and $J(\Q(i))$ has rank at most $2$. Then both ranks must be exactly $2$, and \cite[Proposition 3.8]{schaefer}, shows that $J(K)$ has rank $4$. In fact, one can use the algorithm described in such paper to find four independent $K$-rational points on $J(K)$. Let $O$ be the point at infinity of $C$. Then \cite[Proposition 3.3]{schaefer} shows that $[(0,-3)-O],[(3,6)-O]\in J(K)$ are $\Z[\phi]$-independent, where $\phi$ is the isogeny $1-\zeta_3\colon J\to J$; here $\zeta_3$ is identified with the automorphism $C\to C$ that sends $(x,y)\mapsto (x,\zeta_3y)$. This implies that the points $D_1=[(0,-3)-O],D_2=[(0,-3\zeta_3)-O],D_3=[(3,6)-O],D_4=[(3,6\zeta_3)-O]\in J(K)$ are $\Z$-independent, and thus they form a basis for a finite index subgroup of $J(K)$.

Next, we ought to choose a prime, and we choose $p=13$. This has various pleasant properties: it splits in $K$, $C$ has good reduction at primes above it, $|C(\F_{13})|=22$ and finally no two points of $S$ reduce to the same point modulo a prime above $13$. Since $|S|=22$, this means that if the criterion given by Theorem \ref{chabauty_criterion} is satisfied for every point in $S$, then $S$ is the list of \emph{all} $K$-rational points on $C$. This is quite a lucky coincidence, since usually one is left with $\bmod p$ points that do not come from any known rational point, and needs therefore to prove that there are no rational points with those reductions (for example using the Mordell-Weil sieve as explained in \cite[Section 5]{siksek}).

In order to verify the validity of the criterion, we simply need to compute all objects involved, and then check that it holds at every point of $S$. The code available at \cite{balatu2} shows that a basis of the space of regular differentials for $C$ is

$$(\omega_1,\omega_2,\omega_3)=\left(\frac{1}{y^2}dx,\frac{x}{y^2}dx,\frac{1}{y}dx\right).$$

Notice that these are integral at all places of $K$ above $13$. Now the function \texttt{colema\char`_integrals\char`_on\char`_basis} allows to compute $\int_{D_i}\omega_j$ for every $i\in \{1,\ldots,4\}$ and $j\in\{1,\ldots,3\}$. To complete the $12\times 4$ matrix $T$ given by \eqref{matrixT} it is enough to compute $\int_{D_2'}\omega_j$ and $\int_{D_4'}\omega_j$ for every $j\in \{1,\ldots,3\}$, where $D_2'=[(0,-3\zeta_3^2)-O]$ and $D_4'=[(3,6\zeta_3^2)-O]$, since $\zeta_3$ has only two possible embeddings in $\Q_{13}$.

Now ending the proof is rather easy; what one needs to do is to compute the coefficient $\alpha_1(Q)$ described in Lemma \ref{local_integrals} at every point $Q\in S$ and for all places above $13$. Notice that a well-behaved uniformized at a finite point $(x_0,y_0)$ is just $x-x_0$ and a well-behaved uniformizer at $\infty$ is $x/y$. This gives rise to a $12\times 4$ matrix $A_Q$ as in \eqref{matrixA}. We borrowed some of the Magma code related to \cite{siksek} and available at \cite{siksek2} to compute the matrix $U$ such that $Up^aT$ is in Hermite normal form for some $a$. It turns out that the last $8$ rows of $UT$ are zero, and it is an easy verification that the last two rows of $UA_Q$ have rank $4$ for every $Q\in S$, completing therefore the proof.
\end{proof}

\bibliographystyle{plain}
\bibliography{bibliography}

\end{document}